\title{Picker-Chooser fixed graph games}
\author{
Ma\l gorzata Bednarska-Bzd\c ega
\thanks{Faculty of Mathematics and CS, Adam Mickiewicz University, Pozna\'n, Poland. Email: mbed@amu.edu.pl}
\and Dan Hefetz
\thanks{School of Mathematics, University of Birmingham, Edgbaston, Birmingham B15 2TT, United
Kingdom. Email: danny.hefetz@gmail.com. Research supported by EPSRC grant EP/K033379/1.} \and Tomasz \L uczak
\thanks{Faculty of Mathematics and CS, Adam Mickiewicz University, Pozna\'n, Poland. Email: tomasz@amu.edu.pl. Research supported by NCN grant 2012/06/A/ST1/00261.} 
}

\documentclass[11pt]{article}
\usepackage{amsmath,amssymb,latexsym,color,epsfig,a4,enumerate}

\newif\ifnotesw\noteswtrue

\renewcommand{\deg}{\text{\rm deg}}
\newcommand{\E}{{\mathbb E}}
\newcommand{\inter}{{\mathcal I}}
\newcommand{\eps}{\varepsilon}
\newcommand{\tS}{\tilde{S}_C}

\newtheorem{theorem}{Theorem}[section]
\newtheorem{lemma}[theorem]{Lemma}

\newtheorem{proposition}[theorem]{Proposition}

\newtheorem{corollary}[theorem]{Corollary}
\newtheorem{conjecture}[theorem]{Conjecture}

\newenvironment{proof}{\noindent{\bf Proof\,}}{\hfill$\Box$\medskip}

\begin{document}
\maketitle
  
\begin{abstract}
Given a fixed graph $H$ and a positive integer $n$, a Picker-Chooser $H$-game is a biased game played on the edge set of $K_n$ in which Picker is trying to force many copies of $H$ and Chooser is trying to prevent him from doing so. In this paper we conjecture that the value of the game is roughly the same as the expected number of copies of $H$ in the random graph $G(n,p)$ and prove our conjecture for special classes of graphs $H$ such as complete graphs and trees.
\end{abstract}

\section{Introduction} \label{intro}

A Waiter-Client game is a positional game which was first defined and studied by Beck under the name of Picker-Chooser (see, e.g.~\cite{becksec}). Let $a$ and $b$ be positive integers, let $X$ be a finite set and let ${\mathcal F}$ be a family of subsets of $X$. A biased $(b : 1)$ Waiter-Client game $(X, {\mathcal F})$ is defined as follows. The game proceeds in rounds. In each round, Waiter selects exactly $b+1$ \emph{free} elements of $X$ (that is, elements he has not previously selected) and offers them to Client. Client then selects one of these elements which he keeps and the remaining $b$ elements are claimed by Waiter. If at some point during the game at most $b$ free board elements remain, then all these elements go to Waiter. Waiter's goal is to maximize the number of sets $A \in {\mathcal F}$ whose elements were all claimed by Client by the end of the game, whereas Client aims to minimize this quantity. The set $X$ is referred to as the \emph{board} of the game and the elements of ${\mathcal F}$ are referred to as the \emph{winning sets}. The {\em value} of the game is the number of sets $A \in {\mathcal F}$ whose elements were all claimed by Client by the end of the game, assuming perfect play by both players.     

The interest in such  games is three-fold. Firstly, they are interesting in their own right. For example, the case where Waiter plays randomly is the well-known Achlioptas process (without replacement). Many randomly played Waiter-Client games were considered in the literature, often under different names (see, e.g.~\cite{KLS, KSS, MST}). Secondly, they exhibit a strong probabilistic intuition (see, e.g.~\cite{becksec, TTT}). That is, the outcome of many natural positional games of this type is often roughly the same as it would be had both players played randomly (although, typically, a random strategy for any single player is very far from optimal). The main results of this paper form a natural new example of this intriguing phenomenon. Lastly, it is believed that these games may be useful in the analysis of the so-called \emph{Maker-Breaker} games.

A biased $(1 : b)$ Maker-Breaker game $(X, {\mathcal F})$ is defined as follows. Two players, called Maker and Breaker, take turns in claiming previously unclaimed elements of $X$; usually Maker is the first player. Maker claims exactly one board elements per turn and Breaker claims exactly $b$. Here too the value of the game is the number of sets $A \in {\mathcal F}$ whose elements were all claimed by Maker by the end of the game, assuming perfect play by both players. Maker's goal is to maximize the value of the game, whereas Breaker aims to minimize it.

It was suggested by Beck~\cite{becksec} and subsequently formally conjectured by Csernenszky, M\'andity and Pluh\'ar in~\cite{cdm} that ``being Waiter is not harder than being Maker''. That is, whenever Maker (as the second player) has a winning strategy for the $(1:1)$ Maker-Breaker game $(X, {\mathcal F})$, Waiter has a winning strategy for the $(1:1)$ Waiter-Client game $(X, {\mathcal F})$. Though, in its full generality, this conjecture was recently refuted by Knox~\cite{Knox}, it is still plausible that understanding Waiter-Client games is helpful in the study of Maker-Breaker games. In particular, it was proved in~\cite{pc} that a version of Beck's conjecture which applies to biased games as well, holds in certain special cases. 

We remark that Waiter-Client games are also related to a well-known mis\`ere version of Maker-Breaker games, the so-called Avoider-Enforcer games, in which Enforcer aims to force Avoider to claim as many sets $A \in {\mathcal F}$ as possible (for more information on these games see, for instance,~\cite{HKSae, HKSSae}).

From here on we restrict our attention to fixed graph games. Let $H$ be a graph and let $n$ be a positive integer. The board of the \emph{$H$-game} is the edge set $\binom{[n]}{2}$ of the complete graph on $n$ vertices and the family of winning sets ${\mathcal F}_H$ consists of the edge sets of all copies of $H$ in $K_n$. Let us denote the value of such a biased $(b : 1)$ Waiter-Client game by $S(H,n,b)$ and the value of the analogous Maker-Breaker game by $S_{\textrm{MB}}(H,n,b)$.

Let us first report the known results regarding $S_{\textrm{MB}}(H,n,b)$. Before doing so, let us recall that, as we have already mentioned, often the outcome of a positional game is roughly the same as it would be had both players played randomly. Since the densities of the graphs built by Client and by Maker by the end of the game are the same and are equal to $1/(1+b)$, it would be useful to determine the number of copies of $H$ in the random graph $G(n,1/(1+b))$, where $G(n,p)$ denotes the random graph in which each pair from $\binom{[n]}{2}$ is present independently with probability $p$. It turns out that this number depends mainly on the density of $H$. Hence, let us introduce two measures of density of a graph $H$, both of which are crucial for Waiter-Client $H$-games as well. The maximum density $m(H)$ is defined to be
$$
m(H) = \max \left\{\frac{e(H')}{v(H')} : H' \subseteq H, v(H') \geq 1 \right\} \,, 
$$
where here and throughout the paper $v(G)$ and $e(G)$ denote the number of vertices and edges of $G$ respectively. We shall also use the maximum 2-density $m_2(H)$ of $H$, where 
$$
m_2(H) = \max \left\{\frac{e(H') - 1}{v(H') - 2} : H' \subseteq H, v(H') \geq 3 \right\} \,.
$$

The following result is known (see \cite{B, R}).

\begin{theorem} [\cite{B, R}] \label{th::Gnp}
For every graph $H$ with at least one edge the following holds. If $n p^{m(H)} \to 0$, then a.a.s. $G(n,p)$ contains no copies of $H$. On the other hand, if $n p^{m(H)} \to \infty$, then a.a.s. $G(n,p)$ contains $(c_H + o(1)) n^{v(H)} p^{e(H)}$ copies of $H$, for some constant $c_H > 0$.
\end{theorem}

In~\cite{mbtl} the authors studied the threshold value of $b$ for which $S_{\textrm{MB}}(H,n,b)>0$. It is a straightforward consequence of the probabilistic argument presented there and of Theorem~\ref{th::Gnp}, that the value of the game rapidly grows when $b = \Theta(n^{1/m_2(H)})$. Formally, the following result, which is implicit in~\cite{mbtl}, can be derived.

\begin{theorem}[\cite{mbtl}] \label{th::BL}
For every graph $H$ with  at least three non-isolated vertices there are positive constants $c'$, $c''$, $\beta^-$, and $\beta^+$, such that the following holds. If $b \geq c' n^{1/m_2(H)}$ then 
$$
S_{\textrm{MB}}(H,n,b) = 0 \,.
$$

On the other hand if $b \leq c'' n^{1/m_2(H)}$, then 
$$
\beta^- n^{v(H)} (b+1)^{-e(H)} \leq  S_{\textrm{MB}}(H,n,b) \leq \beta^+  n^{v(H)} (b+1)^{-e(H)} \,.
$$
\end{theorem}

Thus, somewhat unexpectedly, Maker cannot create even a single copy of $H$ until the density of his graph grows to the value which would  guarantee that the number of copies of $H$ in the random graph is as large as the number of its edges. However, soon afterwards, Maker can build roughly the same number of copies of $H$ as the expected number of such copies in $G(n,p)$ with the same density as his graph (i.e. with $p = 1/(b+1)$).

The main purpose of this paper is to show that the behaviour of Waiter-Client $H$-games is quite different and that the value of the game grows almost exactly as suggested by the random graph heuristic. Let us start by stating the following simple corollary of Beck's potential method which we will prove in Section~\ref{sec::trees}.

\begin{theorem} \label{th::BESH}
For every graph $H$  with at least one edge there are positive constants $c_H$ and $c'_H$ such that the following holds.
\begin{description}
\item [(i)] $S(H,n,b) \leq c_H \cdot n^{v(H)} (b+1)^{-e(H)}$ for every $b \geq 1$.
\item [(ii)] If $b \geq c'_H \cdot n^{1/m(H)}$ then $S(H,n,b)=0$.
\end{description}
\end{theorem}

We conjecture that the upper bound on $S(H,n,b)$, given in Theorem~\ref{th::BESH}(i), is tight up to a constant factor, i.e. that the following general conjecture holds for every graph $H$. 

\begin{conjecture} \label{conj}
For every graph $H$ with at least one edge, there are positive constants $c, \alpha^-$ and $\alpha^+$ such that
$$
\alpha^- n^{v(H)} (b+1)^{-e(H)} \leq S(H,n,b) \leq \alpha^+ n^{v(H)} (b+1)^{-e(H)} \,,
$$
provided that $b \leq c \cdot n^{1/m(H)}$. 
\end{conjecture}

This is a rather striking conjecture since, if true, it constitutes a notable example of a game whose value follows the predictions given by a random heuristic so precisely. Nonetheless, we strongly believe it to be true. 

In light of Theorem~\ref{th::BESH}(i), in order to verify Conjecture~\ref{conj}, we need only to prove the lower bound on $S(H,n,b)$, that is, we need to provide a strategy for Waiter which, for appropriate values of $b$, forces Client's graph to contain many copies of $H$. Probably the most natural strategy of this kind is one which constructs $H$ recursively. That is, we choose a suitable subgraph $H'$ of $H$ and in the first stage of the game, we play only on part of the whole board $\binom{[n]}{2}$, building a large number of copies of $H'$. Then we use the remaining pairs to extend the required number of copies of $H'$ to copies of $H$. This method works nicely when the structure of $H$ is simple; we use it in Section~\ref{sec::trees} to prove that Conjecture~\ref{conj} holds for trees.

\begin{theorem} \label{th::correctNumberOfTrees}
Let $k \geq 2$ and $n$ be positive integers and let $T$ be a tree on $k$ vertices. Then there exist positive constants $c$, $\alpha^-$ and $\alpha^+$ which depend on $k$, such that 
$$
\alpha^- n^k \,(b+1)^{1-k} \leq S(T, n, b) \leq \alpha^+ n^k \, (b+1)^{1-k}
$$ 
provided that $b \leq c\cdot n^{k/(k-1)}$. 
\end{theorem}

For more complicated graphs $H$, we develop another method, which is based on counting prohibited structures. Its heart is a rather simple but very useful observation (Theorem~\ref{th::smallFamily}) asserting that Waiter can prevent Client from claiming certain structures if they are very rare. As an immediate consequence of this fact we infer that Conjecture~\ref{conj} holds for every graph $H$ provided that $b$ is not too large, i.e. is bounded from above by the same function as in Theorem~\ref{th::BL}. 

\begin{theorem} \label{th::largeb}
For every graph $H$ with at least three non-isolated vertices there exist positive constants $c$, $\alpha^-$, and $\alpha^+$ such that 
$$
\alpha^- n^{v(H)} (b+1)^{-e(H)} \leq S(H,n,b) \leq \alpha^+  n^{v(H)} (b+1)^{-e(H)} \,,
$$
provided that $b \leq c\cdot n^{1/m_2(H)}$.
\end{theorem} 

In fact, for graphs $H$ that satisfy certain technical conditions which, roughly speaking, assert that $H$ is `well-balanced', we can prove much more. Namely, we can show that not only can Waiter force Client to build many copies of $H$, but he can do it early in the game. Moreover, Waiter can offer pairs such that the following holds: for every pair of vertices $\{v,w\}$, if $v$ and $w$ belong to a copy of $H$ in Client's graph but $\{v,w\}$ is not an edge of $H$, then the pair $\{v,w\}$ belongs to exactly one copy of $H$ and this pair has not yet been offered by Waiter. For such a graph $H$, our strategy for verifying Conjecture~\ref{conj} is as follows. For a given $b = b(n)$, we delete some edges from $H$, thus obtaining a balanced spanning subgraph $H'$ which is sparse enough to ensure $b \leq c n^{1/m_2(H')}$. Applying Theorem~\ref{th::largeb} then forces at least $c'' n^{v(H)} (b+1)^{-e(H')}$ copies of $H'$ in Client's graph. Now, Waiter can offer Client free edges which partly extend his copies of $H'$ to copies of $H$. Clearly, for every edge $e \in E(H) \setminus E(H')$, forcing Client to add $e$ to his copies of $H'$, decreases the number of copies of $H'$ in Client's graph which could potentially be extended further, by a factor of $b+1$. Since in order to complete a copy of $H$ we need to add to $H'$ exactly $e(H) - e(H')$ edges, at the very end we are left with 
$$
\Theta \left(n^{v(H')} (b+1)^{-e(H')} \right) \Theta \left((b+1)^{-(e(H)-e(H'))} \right) = \Theta \left(n^{v(H)} (b+1)^{-e(H)} \right) 
$$
copies of $H$ in Client's graph, as required. This method seems to be quite effective but has one serious drawback -- it only works if we can find a spanning subgraph $H'$ of $H$ which is at the same time sparse and well-balanced. Unfortunately, this is not always possible, and even when it is, the analysis of the structure of $H'$ is often technical and long. Though our method works for a large family of graphs, for the sake of clarity and simplicity, in this paper we consider only complete graphs which suffice to illustrate the technical problems one should overcome in order to apply the method. Note that the maximum density of the complete graph on $k$ vertices is equal to $m(K_k) = (k-1)/2$. 

\begin{theorem} \label{th::manyK}
Let $k \geq 3$ be an integer. Then there exist positive constants $c$, $\alpha^-$ and $\alpha^+$ which depend on $k$, such that 
\begin{description}
\item [(i)] If $k \neq 5$ and $1 \leq b \leq c \cdot n^{2/(k-1)}$, then  
$$
\alpha^- n^k \, (b+1)^{- \binom{k}{2}} \leq S(K_k, n, b) \leq \alpha^+ n^k \, (b+1)^{- \binom{k}{2}} \,.
$$
\item [(ii)] If $k=5$ and $1 \leq b \leq c \cdot n^{2/(k-1)}$, then  
$ S(K_k, n, b) > 0$. 
\end{description}
\end{theorem}

For $k=5$ we can in fact prove a stronger result than the one stated in Theorem~\ref{th::manyK} (ii). However, our general method, which which works for $k \neq 5$ (as well as many other graphs) does not work for $k=5$ and some range of the bias $b$. For simplicity, we decided to state here just the fundamental fact (ii), which, combined with (i), implies that $S(K_k, n, b)$ drops from a positive value to zero around $b = \Theta(n^{2/(k-1)})$ for every $k \geq 3$. Our stronger result for $k = 5$ will be discussed in greater detail in Section~\ref{sec::cliques}.

Finally, let us mention that typically, for any given graph $H$ and any function $b = b(n)$, one can combine the techniques described above, in order to verify Conjecture~\ref{conj} for these $H$ and $b$. For instance, consider the graph $F_9$ on nine vertices which consists of two vertex disjoint copies of $K_4$ joined by a path of length two. Suppose that $b = \Theta \left(n^{1/m(F_9)}\right) = \Theta \left(n^{9/14}\right)$. One can check that the approach based on finding a suitable `well-balanced' subgraph $F' \subseteq F_9$ whose copies can be extended to copies of $F_9$ fails here, because no subgraph $F'$ of $F_9$ fulfills all the technical requirements which are needed to employ this method. Nevertheless, Conjecture~\ref{conj} can still be verified in this case by forcing many `uniformly spread out' copies of $K_4$ on, say, half the vertices and subsequently using the remaining free pairs to join them by 2-paths. Still, it is fairly hard to find and describe a general approach which applies to any graph $H$. Even for relatively small graphs $H$, proving that Conjecture~\ref{conj} holds for these $H$ can be long and technical.   

\subsection{Notation and terminology} \label{sec::prelim}

\noindent Most of our results are asymptotic in nature and whenever necessary we assume that the number of vertices (usually denoted in this paper by $n$ or $s$) is sufficiently large. Our graph-theoretic notation is standard and follows that of~\cite{West}. In particular, we use the following.

For a graph $G$, let $V(G)$ and $E(G)$ denote its sets of vertices and edges respectively, and let $v(G) = |V(G)|$ and $e(G) = |E(G)|$. For disjoint sets $A,B \subseteq V(G)$, let $E_G(A,B)$ denote the set of edges of $G$ with one endpoint in $A$ and one endpoint in $B$ and let $e_G(A,B) = |E_G(A,B)|$. For a vertex $u \in V(G)$ and a set $B \subseteq V(G)$ we abbreviate $E_G(\{u\}, B)$ under $E_G(u, B)$. For a set $S \subseteq V(G)$, let $G[S]$ denote the subgraph of $G$ which is induced on the set $S$. For a vertex $u \in V(G)$ and a set $B \subseteq V(G)$, let $N_G(u, B) = \{v \in B : uv \in E(G)\}$ denote the set of neighbors of $u$ in $B$ and let $d_G(u, B) = |N_G(u, B)|$ denote its \emph{degree} in $B$. We abbreviate $N_G(u, V(G))$ and $d_G(u, V(G))$ under $N_G(u)$ and $d_G(u)$, respectively. Often, when there is no risk of confusion, we omit the subscript $G$ from the notation above. Given two graphs $G$ and $H$ on the same set of vertices $V$, let $G \setminus H$ denote the graph with vertex set $V$ and edge set $E(G) \setminus E(H)$. If $H$ has $n$ vertices, the graph $K_n \setminus H$ is the \emph{complement} of $H$, denoted by $H^c$. 

Assume that some Waiter-Client game, played on the edge set of $K_n$, is in progress. At any given moment during this game, let $G_C$ denote the graph spanned by Client's edges, let $G_W$ denote the graph spanned by Waiter's edges, and let $G_F$ denote the graph spanned by those edges of $K_n$ which are neither in $G_C$ nor in $G_W$. The edges of $G_F$ are called \emph{free}. 

The rest of this paper is organized as follows: in Section~\ref{sec::trees} we prove Theorem~\ref{th::BESH} and Theorem~\ref{th::correctNumberOfTrees}. In Section~\ref{sec::smallFamily} we describe a general efficient strategy for Waiter to avoid rare structures. In Section~\ref{sec::probTools} we prove several properties of a certain model of random graphs. Results obtained in Sections~\ref{sec::smallFamily} and~\ref{sec::probTools} are then used in Section~\ref{sec::eveyH} to estimate the value of Waiter-Client $H$-games; in particular, we prove Theorem~\ref{th::largeb}. We consider the special case in which $H$ is a clique in Section~\ref{sec::cliques}.


\section{Tree games} \label{sec::trees}

As noted in the introduction, our proof of Theorem~\ref{th::correctNumberOfTrees} is based on a simple inductive argument and does not require any of the heavy machinery we develop in this paper. It is thus included in this section, before we describe the tools we need for the much more challenging case where $H$ is a clique. One tool we do already need is Theorem~\ref{th::BESH}. It will be deduced from the following sufficient condition for Client's win in biased Waiter-Client games. 

\begin{proposition} \label{prop::BES}
Let $X$ be a finite set, let ${\mathcal F}$ be a family of subsets of $X$ and let $b$ be a positive integer. Playing the $(b:1)$ Waiter-Client game $(X, \mathcal{F})$, Client has a strategy to ensure that, at the end of the game, the board elements he claimed will span at most $\sum_{A \in \mathcal{F}} (b+1)^{-|A|}$ winning sets.
\end{proposition}

The proof of this proposition is a straightforward application of the potential method, whose details can be found in~\cite{beckwf} (or \cite{TTT}), and is therefore omitted. 

\begin{proof}{\textbf{of Theorem~\ref{th::BESH}}}
Part (i) is an immediate corollary of Proposition~\ref{prop::BES} with $X = E(K_n)$ and $\mathcal{F} = \mathcal{F}_H$.

For Part (ii), let $H'$ be a subgraph of $H$ such that $m(H) = e(H')/v(H')$. Then there exists a positive constant $c$ (depending on $H$) such that if $b > c n^{1/m(H)}$, then 
$$
\sum_{A \in \mathcal{F}_{H'}} (b+1)^{-|A|} \leq n^{v(H')} (b+1)^{-e(H')} < 1 \,.
$$
Applying Proposition~\ref{prop::BES} with $X = E(K_n)$ and $\mathcal{F} = \mathcal{F}_{H'}$ we conclude that, playing a $(b:1)$ Waiter-Client game on $E(K_n)$, Client has a strategy to avoid claiming a copy of $H'$, thereby avoiding claiming a copy of $H$ as well.
\end{proof}

Our next aim is to prove Theorem~\ref{th::correctNumberOfTrees}; it will readily follow from the following two lemmata.

\begin{lemma} \label{lem::manyT1}
Let $T$ be a tree on $k \geq 1$ vertices. If $n$ is sufficiently large and $b \leq n/2^{k+6}$, then $S(T,n,b) \geq 4^{-\binom{k+1}{2}} n^k (b+1)^{1-k}$.
\end{lemma}

\begin{proof}
For every $k \geq 1$, let $t_k(n,b) = 4^{-\binom{k+1}2} n^k (b+1)^{1-k}$. We will prove the lemma by induction on $k$. For $k=1$ the assertion of the lemma is trivially true. Fix some $k \geq 1$ and assume that the assertion of the lemma holds for every tree on $k$ vertices. Let $T$ be an arbitrary tree on $k+1$ vertices; we will prove that if $b \leq n/2^{k+7}$, then Waiter can force Client to build at least $t_{k+1}(n,b)$ copies of $T$. Let $v_{k+1}$ be a leaf of $T$ and let $v_k$ be its unique neighbor. Let $T_k = T \setminus \{v_{k+1}\}$. We partition the vertex set $V(K_n)$ into two subsets $V_1$ and $V_2$ such that $|V_1| = \lceil n/2 \rceil$ and $|V_2| = \lfloor n/2 \rfloor$. Waiter's strategy is divided into two stages.

In the first stage, offering only edges of $K_n[V_1]$, Waiter forces Client to build a family ${\mathcal T}$ consisting of at least $t_k(|V_1|,b)$ copies of $T_k$. This is clearly possible by the induction hypothesis since $b \leq n/2^{k+7} \leq |V_1|/2^{k+6}$. 

In the second stage Waiter forces Client to extend every copy $T^i$ of $T_k$ he has built during the first stage, into many copies of $T$. For every such $T^i$, let $u^i$ denote the vertex which corresponds to $v_k$. Waiter offers $b+1$ edges of $E(u^i, V_2)$ for $\lfloor |V_2|/(b+1) \rfloor$ consecutive rounds. Client is thus forced to build at least $t_k(|V_1|,b) \lfloor |V_2|/(b+1) \rfloor$ copies of $T$. Since 
$$
\left \lfloor \frac{|V_2|}{b+1} \right \rfloor \geq \frac{n/2 - 1}{b+1} - 1 > \frac{n}{4(b+1)}
$$ 
holds for $b < n/4 - 2$, the number of copies of $T$ in Client's graph at the end of the game is at least
$$
t_k(\lceil n/2 \rceil, b) \cdot \frac{n}{4(b+1)} \geq \frac{n^k}{2^k \cdot 4^{\binom{k+1}{2}}(b+1)^{k-1}} \cdot \frac{n}{4(b+1)} \geq \frac{n^{k+1}}{4^{\binom{k+2}{2}}(b+1)^k} = t_{k+1}(n,b) \,,
$$
as claimed.
\end{proof}

\begin{lemma} \label{lem::manyT2}
Let $T$ be a tree on $k \geq 1$ vertices. If $n$ is sufficiently large and $n \leq b \leq n^{k/(k-1)}/2^{k+6}$, then playing a $(b:1)$ Waiter-Client game on $E(K_n)$, Waiter has a strategy to force Client to build at least $4^{-\binom{k+1}{2}} n^k (b+1)^{1-k}$ vertex disjoint copies of $T$. (For $k=1$ we can replace $n^{k/(k-1)}$ with $\infty$.)
\end{lemma}

\begin{proof}
We will prove the lemma by induction on $k$. For $k=1$ the assertion of the lemma is trivially true. Fix some $k \geq 1$ and assume that the assertion of the lemma holds for every tree on $k$ vertices. Let $T$ be an arbitrary tree on $k+1$ vertices and assume that $n \leq b \leq n^{(k+1)/k}/2^{k+7}$. Let $t_k(n,b)$, $v_{k+1}$, $v_k$, $T_k$, $V_1$ and $V_2$ be as in the previous proof. We present a strategy for Waiter to force Client to build at least $t_{k+1}(n,b)$ pairwise vertex disjoint copies of $T$. Waiter's strategy is divided into two stages. 

In the first stage, offering only edges of $K_n[V_1]$, Waiter forces Client to build a family ${\mathcal T}$ consisting of at least $t_k(|V_1|,b)$ pairwise vertex disjoint copies of $T_k$. This is clearly possible by the induction hypothesis since $|V_1| \leq n \leq b \leq n^{(k+1)/k}/2^{k+7} \leq |V_1|^{k/(k-1)}/2^{k+6}$ holds for sufficiently large $n$.

In the second stage Waiter forces Client to extend some of the copies of $T_k$ he has built during the first stage, into a copy of $T$ as follows. Immediately before each round of the second stage, Waiter defines $A$ to be the set of all vertices $u$ which correspond to $v_k$ in a copy of $T_k$ in ${\mathcal T}$, such that $d_{G_C}(u, V_2) = 0$. Moreover, he defines $B$ to be the set of all vertices $v \in V_2$ such that $d_{G_C}(v) = 0$. If $e_{G_F}(A,B) \geq b+1$, then Waiter offers Client $b+1$ arbitrary free edges of $E(A,B)$. The second stage is over as soon as $e_{G_F}(A,B) < b+1$ or $|V_2 \setminus B| \geq t_{k+1}(n,b)$ first holds. 

Note that, at the end of the second stage, every edge $uv \in E(G_C)$ such that $u$ corresponds to $v_k$ in some copy $T' \in {\mathcal T}$ of $T_k$ and $v \in V_2 \setminus B$, extends $T'$ into a copy of $T$ in $G_C$. Moreover, the resulting copies of $T$ are pairwise vertex disjoint. Therefore, in order to complete the proof, it suffices to verify that $|V_2 \setminus B| \geq t_{k+1}(n,b)$ holds at the end of the second stage. 

Suppose for a contradiction that $|V_2 \setminus B| < t_{k+1}(n,b)$ holds at the end of the second stage. It is not hard to see that if a vertex $u$ corresponds to $v_k$ in some copy of $T_k$ in ${\mathcal T}$, then either $u \in A$ or $d_{G_C}(u, V_2) = 1$. Similarly, for every $v \in V_2$, either $v \in B$ or $d_{G_C}(v) = 1$. Therefore, the total number of rounds played in the second stage is $|V_2 \setminus B| = |{\mathcal T}| - |A|$; denote this number by $r$. Since $|{\mathcal T}| \geq t_k(|V_1|,b)$ holds by the induction hypothesis, we obtain 
\begin{eqnarray} \label{eq::lowerBoundA}
|A| &=& |{\mathcal T}| - r > |{\mathcal T}| - t_{k+1}(n,b) \geq t_k(n/2,b) - t_{k+1}(n,b) \nonumber \\
&=& \frac{n^k}{4^{\binom{k+1}{2}}(b+1)^{k-1}} \left(\frac{1}{2^k} - \frac{n}{4^{k+1}(b+1)}\right) \geq \frac{n^k}{4^{\binom{k+1}{2}}(b+1)^{k-1}} \cdot \frac{1}{2^{k+1}} \,,
\end{eqnarray}
where the last inequality holds since $b \geq n$.

Similarly
\begin{eqnarray} \label{eq::lowerBoundB}
|B| &=& |V_2| - |V_2 \setminus B| > |V_2| - t_{k+1}(n,b) > n/2 - 1 - t_{k+1}(n,b) \nonumber \\ &=& n/2 - 1 - \frac{n^{k+1}}{4^{\binom{k+2}{2}}(b+1)^k} > 3n/8 \,,
\end{eqnarray}
where the last inequality holds since $b \geq n$.

It follows by our assumption that $|V_2 \setminus B| < t_{k+1}(n,b)$ and by the description of the second stage of Waiter's strategy, $|A||B| - r (b+1) \leq e_{G_F}(A,B) < b+1$ holds at the end of this stage. Hence, using~\eqref{eq::lowerBoundA} and~\eqref{eq::lowerBoundB} we obtain 
\begin{equation} \label{eq::lowerBoundr}
r+1> \frac{|A||B|}{b+1} > \frac{n^k}{4^{\binom{k+1}{2}}(b+1)^{k-1}} \cdot \frac{1}{2^{k+1}} \cdot \frac{3n}{8} \cdot \frac{1}{b+1} \geq \frac{1.5 n^{k+1}}{4^{\binom{k+2}{2}}(b+1)^{k}} = 1.5 t_{k+1}(n,b)\,.
\end{equation}

Since $t_{k+1}(n,b) > 2$ for every $b \leq n^{(k+1)/k}/2^{k+7}$, it follows by~\eqref{eq::lowerBoundr} that $r > t_{k+1}(n,b)$. Since $r = |V_2 \setminus B|$, this contradicts our assumption that $|V_2 \setminus B| < t_{k+1}(n,b)$.    
\end{proof}

\begin{proof}{\textbf{of Theorem~\ref{th::correctNumberOfTrees}}}
The required upper bound on $S(T, n, b)$ follows immediately from Theorem~\ref{th::BESH}(i); it thus remains to prove the lower bound.

For $b \leq n/2^{k+6}$ the desired lower bound follows from Lemma~\ref{lem::manyT1} and for $n \leq b \leq n^{k/(k-1)}/2^{k+6}$ it follows from Lemma~\ref{lem::manyT2}. Assume then that $n/2^{k+6} < b < n$. Observe that Waiter-Client games are bias-monotone, that is, if Waiter can force Client to fully claim, say, $t$ winning sets in a $(b' : 1)$ game, then Waiter can achieve the same goal if his bias is smaller than $b'$. We conclude that $S(T, n, b) \geq S(T, n, n) = \Omega \left(n^k \cdot (n+1)^{1-k} \right) = \Omega \left(n^k \cdot (b+1)^{1-k} \right)$.
\end{proof}

\section{Big Family Theorem} \label{sec::smallFamily}
  
In this section we state and prove the main game theoretic tool of this paper, which is also of independent interest. Roughly speaking, it asserts that if almost every $M$-subset of the board is ``good'', then Waiter can force Client to claim such a set in $M$ rounds. The main idea of our proof of Theorem~\ref{th::manyK} is to show that almost every graph Client might end up with, contains the correct number of copies of $K_k$ and then to apply our game theoretic tool to conclude that Waiter can force Client to build one of these graphs.      
  
\begin{theorem} \label{th::smallFamily}
Let $X$ be a set of size $N$, let ${\mathcal H}$ be an $M$-uniform family of subsets of $X$ and let $b = b(N)$ be a positive integer. Let $0 < \alpha < 1$ be a real number such that $|{\mathcal H}| \geq (1-\alpha^M) \binom{N}{M}$ and $b+1 \leq (1-\alpha) N/M$. Then, playing the $(b:1)$ Waiter-Client game $(X, {\mathcal H})$, Waiter has a strategy to force Client to fully claim some $A \in {\mathcal H}$ during the first $M$ rounds of the game.
\end{theorem}

\begin{proof}
Let ${\mathcal H}$ be the family of sets described in the theorem and let ${\mathcal H}^c$ be its $M$-uniform complement, that is, ${\mathcal H}^c = \left\{A \in \binom{X}{M} : A \notin {\mathcal H} \right\}$. It follows by the assumption of the theorem that $|{\mathcal H}^c| \leq \alpha^M \binom{N}{M}$. In order to prove the theorem it suffices to prove that Waiter can prevent Client from claiming some $A \in {\mathcal H}^c$ during the first $M$ rounds of the game.    

Let $C_0 = W_0 = \emptyset$ and, for every positive integer $i$, let $C_i$ and $W_i$ denote the set of elements of Client and of Waiter respectively, immediately after the $i$th round. Moreover, for every non-negative integer $i$ let
$$
E_i = \{A \setminus C_i : A \in {\mathcal H}^c, A \cap W_i = \emptyset \textrm{ and } C_i \subseteq A\} \,.
$$

For every $x \in X$ and every non-negative integer $i$ let  
$$
\deg_i(x) = |\{S \in E_i : x \in S\}|
$$
and let $N_i = N - i(b+1)$. Note that $N_i$ is the number of free board elements immediately after round $i$.

The definition of $E_i$ implies that 
\begin{equation} \label{eq::sizeOfEdge}
|S| = M-i \textrm{ holds for every } S \in E_i. 
\end{equation}

Moreover, if in the $i$th round Client claimed $y$, then $S \cup \{y\} \in E_{i-1}$ whenever $S \in E_i$ and thus 
\begin{equation} \label{eq::degree}
|E_i| \leq \deg_{i-1}(y) \,. 
\end{equation}

We are now ready to describe Waiter's strategy. For every positive integer $i$, in the $i$th round, Waiter offers Client exactly $b+1$ free board elements $x$ whose value of $\deg_{i-1}(x)$ is minimal (breaking ties arbitrarily). It remains to prove that this is a winning strategy.

Let $x$ be an element Waiter has offered Client in the $i$th round. It follows by Waiter's strategy that  
\begin{equation} \label{eq::average}
\deg_{i-1}(x) \leq \frac{1}{N_{i-1} - b} \sum_{u \in X \setminus (C_{i-1} \cup W_{i-1})} \deg_{i-1}(u) = \frac{|E_{i-1}| (M - (i-1))}{N_{i-1} - b} \,,
\end{equation}
where the last equality follows since, for every $S \in E_{i-1}$, $|S| = M - (i-1)$ holds by~\eqref{eq::sizeOfEdge} and because $S \subseteq X \setminus (C_{i-1} \cup W_{i-1})$ holds by the definition of $E_{i-1}$. 

Hence, regardless of which element $y$ Client claims in round $i$, it follows by~\eqref{eq::degree} and~\eqref{eq::average} that 
\begin{equation} \label{eq::DecreasingSize}
|E_i| \leq \deg_{i-1}(y) \leq \frac{|E_{i-1}| (M - (i-1))}{N_{i-1} - b} \,.
\end{equation}

Consequently, since $|E_0| = |{\mathcal H}^c| \leq \alpha^M \binom{N}{M}$, by an iterated application of~\eqref{eq::DecreasingSize} we obtain 
\begin{eqnarray*}
|E_M| &\leq& \prod_{i=0}^{M-1} \frac{M-i}{N_i - b}\, |E_0| \\
&\leq& \frac{M!}{(N_{M-1} - b)^M}\, |E_0| \\
&\leq& \frac{M!}{(N - M(b+1))^M}\, \alpha^M \binom{N}{M} \\
&<& \left(\frac{\alpha N}{N - M(b+1)}\right)^M \\
&\leq& 1 \,,
\end{eqnarray*}
where in the last inequality we used the assumption that $b+1 \leq (1-\alpha) N/M$.

Hence $|E_M| < 1$, which means that Client did not claim all elements of any $A \in {\mathcal H}^c$ in the first $M$ rounds and so Waiter has achieved his goal. 
\end{proof}

\section{Probabilistic tools} \label{sec::probTools}

This section contains several results of different levels of difficulty. It is thus divided into two subsections: the first containing some useful terminology and simple facts about certain models of random graphs, and the second containing more advanced results, regarding the number of copies of a fixed graph in those random graphs.

\subsection{Preliminaries}

We begin this section with some more notation and terminology. A graph $H$ on at least 3 vertices is called \emph{$m_2$-balanced} if $m_2(H) = (e(H)-1)/(v(H)-2)$. 

We introduce two additional graph invariants $g_1(H)$ and $g_2(H)$ of a graph $H$, which will be used several times in the remainder of the paper. These invariants depend on the following two families of subgraphs of a graph $H$: 
\begin{eqnarray*}
\inter_H&=&\{H'\subseteq H\colon H'\neq H\text{ and $H'$ is a clique\,}\}\,,\\
\inter^c_H&=&\{H'\subseteq H\colon H'\notin\inter_H\,\}\,.  
\end{eqnarray*}
Given a graph $H$ with at least two edges let 
$$
g_1(H)= \max \left\{\frac{v(H)-v(H')}{e(H)-e(H')}\,\colon H'\in\inter^c_H \ \text{and}\ (2 \leq e(H') < e(H) \textrm{ or } (e(H') = 0 \text{ and } v(H') = 2))\right\},
$$
and
$$
g_2(H)= \min \left\{\frac{v(H')-2}{e(H')-1}\,\colon e(H')\ge 2 \ \text{and}\ (H'\in\inter_H \textrm{ or } H'=H)\right\}.
$$

The main objects of study in this section are random graph models which are slightly different from the classic $G(n,p)$ and $G(n,M)$ models. For a graph $H$ on the vertex set $\{v_1, \ldots, v_t\}$ we define a graph ${\mathbb B}_H(V_1, \ldots, V_t;n)$, called the \emph{$n$th-blow-up of $H$}, as follows. We replace every vertex $v_i$ of $H$ with a set $V_i$ of $n$ isolated vertices and every edge $v_i v_j$ of $H$ with the corresponding complete bipartite graph, that is, with the set of edges $\{xy : x \in V_i, y \in V_j\}$. Let $H'$ be a subgraph of $H$ on the vertex set $\{v_{i_1}, \ldots, v_{i_r}\}$. A copy of $H'$ in ${\mathbb B}_H(V_1, \ldots, V_t;n)$ is said to be \emph{canonical} if $v_{i_j} \in V_{i_j}$ for every $1 \leq j \leq r$. 

Let ${\mathbb G}(H,n,p)$ denote a graph obtained by randomly selecting every edge of the $n$th-blow-up of $H$ with probability $p$, independently of all other edge selections.     

For a graph $H$ and its subgraph $H'$ let $Y_{H'}$ be the random variable counting the number of canonical copies of $H'$ in ${\mathbb G}(H,n,p)$; note that ${\mathbb E}(Y_{H'}) = n^{v(H')} p^{e(H')}$. Let
$$
\hat{f}_H(n,p) = \min \left\{{\mathbb E}(Y_{H'}) : H' \subseteq H, \, e(H') \geq 1 \right\} 
$$
and let
$$
f_H(n,p) =  \min \left\{{\mathbb E}(Y_{H'}) : H'\in\inter^c_H,\, e(H')\geq 2\right\}\,.
$$

The following two simple lemmata provide an alternative characterization of $m_2$-balanced graphs and describe some useful properties of $g_1$, $g_2$,  $\hat{f}_H(n,p)$ and $f_H(n,p)$. 

\begin{lemma} \label{lem::m2balanced}\
\begin{enumerate}[(i)]
\item\label{m2bal}
A graph $H$ with at least two edges and no isolated vertices is $m_2$-balanced if and only if 
$$
\frac{v(H)-v(H')}{e(H)-e(H')} \leq \frac{v(H)-2}{e(H)-1}
$$
for every $H' \subsetneq H$ with $e(H') \geq 2$.
\item\label{g1less}
If $g_1(H) < g_2(H)$ and there exists a connected graph $H' \in \inter^c_H$ such that $2 \leq e(H') < e(H)$, then $g_2(H) < 1$. 
\item\label{g2m2}
If $H$ has at least two edges and $g_1(H)\le g_2(H)$, then $g_2(H)=1/m_2(H)$.
\end{enumerate}
\end{lemma}

\begin{proof}
Starting with (\ref{m2bal}), let $H' \subsetneq H$ be an arbitrary subgraph with $e(H') \geq 2$. A straightforward calculation shows that
$$
\frac{v(H)-v(H')}{e(H)-e(H')} \leq \frac{v(H)-2}{e(H)-1} \,\, \Longleftrightarrow \,\, \frac{e(H') - 1}{v(H') - 2} \leq \frac{e(H) - 1}{v(H) - 2} \,.
$$
Since the right hand side of the above equivalence clearly holds for any $H' \subsetneq H$ with $e(H') \leq 1$, the assertion of the lemma follows by the definition of the maximum 2-density.

Now we prove (\ref{g1less}). Assume that $g_1(H) < g_2(H)$ and let $H' \in \inter^c_H$ be a connected graph satisfying $2 \leq e(H') < e(H)$. Then
$$   
\frac{v(H)-v(H')}{e(H)-e(H')}<\frac{v(H)-2}{e(H)-1}\,,
$$
which is equivalent to
$$
\frac{v(H)-2}{e(H)-1}<\frac{v(H')-2}{e(H')-1}\,.
$$
Since $H'$ is connected, the right hand side of the above inequality is at most $1$. Therefore 
$$
g_2(H)\le \frac{v(H)-2}{e(H)-1}<1\,.
$$

It remains to prove (\ref{g2m2}). Assume that $g_1(H)\le g_2(H)$. Then for every $H'\in\inter^c_H$ with $2\le e(H')<e(H)$ we have
$$   
\frac{v(H)-v(H')}{e(H)-e(H')}\le\frac{v(H)-2}{e(H)-1}\,,
$$
which is equivalent to
$$
\frac{v(H)-2}{e(H)-1}\le\frac{v(H')-2}{e(H')-1}\,.
$$
Since, moreover, $H \in \inter^c_H$, it follows that
\begin{equation}\label{m2a}
\frac{v(H)-2}{e(H)-1} = \min_{H' \in \inter^c_H, \atop e(H') \geq 2} \frac{v(H')-2}{e(H')-1} \,.
\end{equation}
The definition of $g_2(H)$ and~\eqref{m2a} imply that 
$$
g_2(H) = \min_{H' \subseteq H, \atop e(H') \geq 2} \frac{v(H')-2}{e(H')-1} = \frac{1}{m_2(H)} \,.
$$
\end{proof}

\begin{lemma} \label{lem::propf}  
Let $H$ be a graph with at least two edges and no isolated vertices and let $s$ be a positive integer.
\begin{enumerate}[(i)]
\item\label{hatdense} 
If $p \geq s^{-1/m_2(H)}$, then $\hat{f}_H(s,p) = s^2 p$.
\item\label{fg1} 
Let $c_0 \leq 1$ be a positive constant. If $p \leq c_0^{-1} s^{-g_1(H)}$, then $f_H(s,p) \geq c_0^{e(H)} s^{v(H)} p^{e(H)}$.
\item\label{fm2} 
Let $c_0 \leq 1$ be a positive constant. If $p \leq c_0^{-1} s^{-1/m_2(H)}$ and $H$ is $m_2$-balanced, then $f_H(s,p) \geq c_0^{e(H)} s^{v(H)} p^{e(H)}$.
\end{enumerate}
\end{lemma}

\begin{proof} Starting with (\ref{hatdense}), if $H'$ is a subgraph of $H$ consisting of a single edge, then ${\mathbb E}(Y_{H'}) = s^2 p$. Thus $\hat{f}_H(s,p) \leq s^2 p$. Conversely, let $H'$ be a subgraph of $H$ such that $\hat{f}_H(s,p) = s^{v(H')} p^{e(H')}$ and suppose for a contradiction that $s^{v(H')} p^{e(H')} < s^2 p$. Then
$$
p < s^{- \frac{v(H') - 2}{e(H') - 1}} \leq s^{-1/m_2(H)} \,,
$$
contrary to our assumption.   

Next, we prove (\ref{fg1}) and (\ref{fm2}). Let $H_0\in\inter^c_H$ be a subgraph of $H$ such that $e(H_0) \geq 2$ and $f_H(s,p) = s^{v(H_0)} p^{e(H_0)}$. If $H_0 = H$, then our assertions clearly hold (with $c_0 = 1$); thus, assume that $H_0 \neq H$. If $p \leq  c_0^{-1} s^{-g_1(H)}$, then $p \leq c_0^{-1} s^{(v(H_0)-v(H))/(e(H)-e(H_0))}$ holds by the definition of $g_1$. Similarly, if $p \leq  c_0^{-1} s^{-1/m_2(H)}$ and $H$ is $m_2$-balanced, then $p \leq c_0^{-1} s^{(v(H_0)-v(H))/(e(H)-e(H_0))}$ holds by Lemma~\ref{lem::m2balanced}(\ref{m2bal}). Hence, in both cases (ii) and (iii), we have $s^{v(H_0)} p^{e(H_0)} \geq c_0^{e(H)-e(H_0)} s^{v(H)} p^{e(H)}$. Since $c_0 \leq 1$ by assumption, this concludes the proof of the lemma.   
\end{proof}

Until now, we have discussed the random graph model ${\mathbb G}(H,n,p)$ which is reminiscent of $G(n,p)$. However, our Big Family Theorem (Theorem~\ref{th::smallFamily}) applies to graphs with a prescribed number of edges. Hence, we will now introduce another random graph model ${\mathbb G}(H,n,M)$ which is reminiscent of $G(n,M)$. Let ${\mathbb G}(H,n,M)$ denote a graph obtained by selecting uniformly at random precisely $M$ edges of the $n$th-blow-up of $H$. For a graph $H$ and its subgraph $H'$ let $X_{H'}$ be the random variable counting the number of canonical copies of $H'$ in ${\mathbb G}(H,n,M)$. For every graph $H$ with at least one edge let
$$
\hat{f}_H(n,M) = \min \left\{{\mathbb E}(X_{H'}) : H' \subseteq H, \, e(H') \geq 1 \right\}
$$
and for every graph $H$ with at least two edges let
$$
f_H(n,M) = \min \left\{{\mathbb E}(X_{H'}) : H'\in\inter^c_H,\, e(H')\geq 2\right\}\,.
$$

There is a well-known asymptotic equivalence between the random graph models $G(n,p)$ and $G(n,M)$ and an analogous equivalence between ${\mathbb G}(H,n,p)$ and ${\mathbb G}(H,n,M)$ can be established as well. In this paper, we will prove results for the more convenient model ${\mathbb G}(H,n,p)$ and then transfer them (implicitly) to the model ${\mathbb G}(H,n,M)$ using Pittel's inequality (see, e.g.,~\cite{JLR}, (1.6) on page 17). In particular, we will use the fact that, for $p = M/(e(H) n^2)$, we have $\hat{f}_H(n,M) = \Theta(\hat{f}_H(n,p))$ and $f_H(n,M) = \Theta(f_H(n,p))$. This holds since, for every graph $H$ there exists a constant $c = c(H)$ such that for every sufficiently large integer $n$, $M = M(n) > c$, and $p = M/(e(H) n^2)$, we have ${\mathbb E}(X_{H}) \leq {\mathbb E}(Y_{H}) \leq 2 {\mathbb E}(X_{H})$.

\subsection{Counting copies of $H$ in ${\mathbb G}(H,n,M)$}

Our first result in this section asserts that the probability that ${\mathbb G}(H,n,M)$ contains too few canonical copies of $H$ is exponentially small. 


\begin{lemma} \label{lem::denseM}
Let $H$ be a graph, $M = M(n) = \omega(1)$ and $M \leq e(H) n^2/2$. 
Then there exists a constant $c > 0$ such that the probability that there are at least ${\mathbb E}(X_H)/2$ canonical copies of $H$ in ${\mathbb G}(H,n,M)$ is at least $1 - \exp(-c \hat{f}_H(n,M))$.  
\end{lemma}

As noted above, Lemma~\ref{lem::denseM} is an immediate corollary of its ${\mathbb G}(H,n,p)$ analogue which can be stated as follows. 

\begin{lemma} \label{lem::denseP}
Let $H$ be a graph, $p = p(n)=\omega(n^{-2})$ and $p\leq 1/2$. Then there exists a constant $c > 0$ such that the probability that there are at least ${\mathbb E}(Y_H)/2$ canonical copies of $H$ in ${\mathbb G}(H,n,p)$ is at least $1 - \exp(-c \hat{f}_H(n,p))$.  
\end{lemma}

In the proof of Lemma~\ref{lem::denseP} we will make use of the following concentration inequality.

\begin{theorem} [Theorem 2.14 in~\cite{JLR}] \label{th::214}
Let $\Gamma$ be a finite set, let $S$ be a family of subsets of $\Gamma$ and let $\Gamma_p$ be a random set obtained from $\Gamma$ by selecting every element of $\Gamma$ independently, with probability $p$. For every $A \in S$, let $I_A$ denote the indicator random variable for the event  $A \subseteq \Gamma_p$. Let $X = \sum_{A \in S} I_A$ and let $\bar\Delta = \sum \sum_{A \cap B \neq \emptyset} {\mathbb E}(I_A I_B)$. Then for $0 \leq t \leq {\mathbb E}(X)$ we have
$$
\Pr(X \leq {\mathbb E}(X) - t) \leq \exp \left(- \frac{t^2}{2 \bar\Delta} \right) \,.
$$
\end{theorem} 

\begin{proof}{\textbf{of Lemma~\ref{lem::denseP}}}
Let $S = \{H_1, \ldots, H_m\}$ be the family of all canonical copies of $H$ in the $n$th-blow-up of $H$. For every $1 \leq i \leq m$, let $I_i$ be the indicator random variable for the event $H_i \subseteq {\mathbb G}(H,n,p)$; then $Y_H = \sum_{i=1}^m I_i$. A straightforward calculation shows that $\bar\Delta := \sum \sum_{H_i \cap H_j \neq \emptyset} {\mathbb E}(I_i I_j) \leq {\mathbb E}(Y_H)^2/\hat{f}_H(n,p)$. Hence, applying Theorem~\ref{th::214} with $t = {\mathbb E}(Y_H)/2$, we obtain 
$$
\Pr(Y_H \leq {\mathbb E}(Y_H)/2) \leq \exp \left(- \frac{\hat{f}_H(n,p)}{8} \right) \,.
$$
\end{proof}

Before we state and prove the main result of this section, we need one more definition. We say that a family ${\mathcal F}$ of subgraphs of ${\mathbb B}_H(V_1, \ldots, V_t;n)$ is a \emph{sparse $H$-family} if the following two conditions hold:
\begin{description} 
\item [(i)] Every $G \in {\mathcal F}$ is a canonical copy of $H$. 
\item [(ii)] Any two distinct graphs $G_1, G_2 \in {\mathcal F}$ are either disjoint or
$G_1 \cap G_2\in\inter_H$. 
\end{description}
It turns out that, for the right values of $M$, with very high probability, ${\mathbb G}(H,n,M)$ contains a large sparse $H$-family. 

\begin{lemma} \label{lem::largeSparseFamily}
For every $\varepsilon > 0$ and every graph $H$ with at least two edges, there exist positive constants $\alpha < 1$, $\beta$ and $\delta$ such that the following holds. For every $n$ and $M$ such that
$$
M\geq\max\{\delta n,\,\varepsilon n^{2-1/m_2(H)}\}\text{ and }f_H(n,M) \leq \delta n^2,
$$ 
the probability that ${\mathbb G}(H,n,M)$ contains a sparse $H$-family with at least $\beta f_H(n,M)$ copies of $H$ is greater than $1 - \alpha^M$.
\end{lemma}

Since  $f_H(n,M)$ is of  the same order as $f_H(n,p)$ and since, by Chernoff's bound, the number of edges in the binomial random graph ${\mathbb G}(H,n,p)$ is sharply concentrated around its expectation, Lemma~\ref{lem::largeSparseFamily} is a straightforward corollary of its binomial analogue which can be stated as follows.

\begin{lemma} \label{lem::largeSparseFamilyBinomial}
For every $\varepsilon > 0$ and every graph $H$ with at least two edges, there exist positive constants $\alpha < 1$, $\beta$ and $\delta < 1$ such that the following holds. For every $n$ and $p$ such that
$$
p \geq \max\{(\delta n)^{-1},\,\varepsilon n^{-1/m_2(H)}\}\text{ and }f_H(n,p) \leq \delta n^2,
$$ 
the probability that ${\mathbb G}(H,n,p)$ contains a sparse $H$-family with at least $\beta f_H(n,p)$ copies of $H$ is greater than $1 - \alpha^{n^2 p}$.
\end{lemma}

In the proof of Lemma~\ref{lem::largeSparseFamilyBinomial} we will use the following well-known concentration inequality due to Talagrand~\cite{Talagrand}.

\begin{theorem} [Theorem 2.29 in~\cite{JLR}] \label{th::Talagrand}
Suppose that $Z_1, \ldots, Z_N$ are independent random variables taking their values in the set $\{0,1\}$. Suppose further that $X = f(Z_1, \ldots, Z_N)$, where $f : \{0,1\}^N \to {\mathbb R}$ is a function such that there exist constants $c_1, \ldots, c_N$ and a function $\psi : {\mathbb R}\to {\mathbb R}$ for which the following two conditions hold:
\begin{description}
\item [(a)] If $z, z' \in \{0,1\}^N$ differ only in the $k$th coordinate, then $|f(z') - f(z)| \leq c_k$.
\item [(b)] If $z \in \{0,1\}^N$, $r \in {\mathbb R}$ and $f(z) \geq r$, 
then there exists a set $J \subseteq \{1, \ldots, N\}$ with $\sum_{i \in J} c_i^2 \leq \psi(r)$, such that for all $y \in \{0,1\}^N$ with $y_i = z_i$ for every $i \in J$, we have $f(y) \geq r$.
\end{description}
Then for every $r \in {\mathbb R}$ and $t \geq 0$ we have
$$
\Pr (X \leq r-t) \Pr(X \geq r) \leq \exp(-t^2/(4\psi(r))) \,.
$$
\end{theorem} 

While proving Lemma~\ref{lem::largeSparseFamilyBinomial}, for a constant $C > 0$, we will find a large family ${\mathcal A}$ of canonical copies of $H$ in ${\mathbb G}(H,n,p)$ which satisfies the following five properties:

\begin{enumerate}[{\bf (P1)}]
\item\label{no_edges}
 The number of edges in the union of all graphs in the family ${\mathcal A}$ is at most $C n^2 p$. 
\item\label{few_copies_e}
Every edge of ${\mathbb G}(H,n,p)$ belongs to at most $C f_H(n,p)/(n^2p)$ graphs of ${\mathcal A}$.
\item\label{nonedge} 
For all graphs $H_1, H_2 \in {\mathcal A}$, if $V(H_1 \cap H_2) = \{x,y\}$ for some $x, y \in V({\mathbb G}(H, n, p))$, then $xy \in E(H_1 \cap H_2)$. 
\item\label{sparse_inters}
If $H_1 \in {\mathcal A}$ and $F$ is an induced subgraph of $H_1$ such that $v(F) \geq 3$ and $e(F) \leq 1$, then $H_1 \cap G \neq F$ for every $G \in {\mathcal A} \setminus \{H_1\}$.
\item\label{few_copies_2e}
If $H_1 \in {\mathcal A}$ and $F$ is an induced subgraph of $H_1$ such that $e(F) \geq 2$ and $F\in\inter^c_H$, then there are at most $C$ graphs $G \in {\mathcal A}\setminus \{H_1\}$ for which $H_1 \cap G = F$. 
\end{enumerate}

In the remainder of this section, for a positive constant $C$, we will denote by $S_C(n,p)$ the size of a  largest family of canonical copies of $H$ in ${\mathbb G}(H,n,p)$ which satisfies all of the properties (P\ref{no_edges})--(P\ref{few_copies_2e}).

The main ingredient of our proof of Lemma~\ref{lem::largeSparseFamilyBinomial} is showing that $S_C(n,p)$ is large. We will first prove that this is true in expectation.   

\begin{lemma} \label{lem::fnp}
For every $\varepsilon > 0$ and every graph $H$ with at least two edges, there exist positive constants $C$ and $\delta<1$ such that the following holds. For every $n$ and $p$ such that
$$
p \geq \max\{(\delta n)^{-1},\,\varepsilon n^{-1/m_2(H)}\}\text{ and }f_H(n,p) \leq \delta n^2,
$$  
we have
$${\mathbb E}(S_C(n,p)) \geq f_H(n,p)/2.$$
\end{lemma}

\begin{proof}
We denote by $\tS(n,p)$ the size of a  largest family of canonical copies of $H$ in ${\mathbb G}(H,n,p)$ which satisfies all of the properties (P\ref{few_copies_e})--(P\ref{few_copies_2e}). Our first goal is to prove that $\E(\tS(n,p))\geq  2f_H(n,p)/3$.

Let $\rho = \rho(n,p) = f_H(n,p)/{\mathbb E}(Y_H)$. We `accept' every canonical copy of $H$ in ${\mathbb G}(H,n,p)$ independently with probability $\rho$. Let ${\mathcal A}$ denote the family of accepted copies of $H$ and let  $Z_H=|{\mathcal A}|$. Clearly ${\mathbb E}(Z_H) = f_H(n,p)$. Now we delete (deterministically) some copies from ${\mathcal A}$ so that the remaining family satisfies Properties (P\ref{few_copies_e})--(P\ref{few_copies_2e}). More precisely, we delete from ${\mathcal A}$  a copy $H_1$ if and only if at least one of the following Properties is satisfied.

\begin{enumerate}[{\bf (P1$'$)}]
 \setcounter{enumi}{1}
\item\label{few_copies_e'}
$H_1$ shares an edge with more than $C f_H(n,p)/(n^2 p)$ graphs of ${\mathcal A}$.
\item\label{nonedge'} 
There exists $H_2 \in {\mathcal A}$ such that $V(H_1 \cap H_2) = \{x,y\}$ for some $x, y \in V({\mathbb G}(H, n, p))$ and $xy \notin E(H_1 \cap H_2)$. 
\item\label{sparse_inters'}
There exist an induced subgraph $F$ of $H_1$ and $G \in {\mathcal A} \setminus \{H_1\}$ such that $v(F) \geq 3$,  $e(F) \leq 1$ and $H_1 \cap G = F$.
\item\label{few_copies_2e'}
There exists an induced subgraph  $F$ of $H_1$ such that $e(F) \geq 2$, $F\in\inter^c_H$ and there are more than $C$ graphs $G \in {\mathcal A}\setminus \{H_1\}$ for which $H_1 \cap G = F$. 
\end{enumerate}

We will prove that the expected number of the remaining copies in ${\mathcal A}$ is at least $2f_H(n,p)/3$, provided that $C$ is sufficiently large. 

Starting with (P\ref{few_copies_2e'}$'$), let $H_1$ be an arbitrary canonical copy of $H$ in ${\mathbb G}(H,n,p)$ and let $F$ be an induced subgraph of $H_1$, where $e(F) \geq 2$ and $F\in\inter^c_H$. The expected number of accepted copies of $H$ whose intersection with $H_1$ is precisely $F$ is at most
\begin{equation} \label{eq::v}
\rho n^{v(H) - v(F)} p^{e(H) - e(F)} = \frac{\rho {\mathbb E}(Y_H)}{{\mathbb E}(Y_F)} = \frac{f_H(n,p)}{{\mathbb E}(Y_F)} \leq 1 \,,
\end{equation}
where the last inequality follows by the definition of $f_H(n,p)$. 

Let $A_F$ denote the event that there are more than $C$ accepted copies of $H$ whose intersection with $H_1$ is precisely $F$. It follows by~\eqref{eq::v} and by Markov's inequality that $\Pr(A_F) < 1/C$. Summing over all choices of $H_1$ and $F$ as above, we conclude that the expected number of accepted copies of $H$ which intersect more than $C$ other accepted copies of $H$ on a given induced subgraph $F$ with at least two edges is at most  
$$
{\mathbb E}(Z_H) 2^{v(H)} /C = f_H(n,p) 2^{v(H)}/C \leq f_H(n,p)/12 \,,
$$ 
where the last inequality holds for sufficiently large $C$. 

Next, we consider (P\ref{sparse_inters'}$'$). Given a canonical copy $H_1$ of $H$ in ${\mathbb G}(H,n,p)$ and an induced subgraph $F$ of $H_1$ such that $v(F) \geq 3$ and $e(F) \leq 1$, the expected number of accepted copies of $H$ whose intersection with $H_1$ is precisely $F$ is at most 
\begin{equation} \label{eq::iv}
\rho n^{v(H) - 3} p^{e(H) - 1} = \frac{f_H(n,p)}{{\mathbb E}(Y_H)} \cdot \frac{{\mathbb E}(Y_H)}{n^3 p} = \frac{f_H(n,p)}{n^3 p} \leq \delta^2 \,, 
\end{equation}
where the last inequality holds since $f_H(n,p) \leq \delta n^2$ and $p \geq (\delta n)^{-1}$ by assumption.

Summing over all choices of $H_1$ and $F$ as above, we conclude that the expected number of accepted copies of $H$ which intersect another accepted copy of $H$ on a given induced subgraph $F$ with at least three vertices and at most one edge is at most
$$
{\mathbb E}(Z_H) 2^{v(H)} \delta^2 = f_H(n,p) 2^{v(H)} \delta^2 \leq f_H(n,p)/12 \,,  
$$
where the last inequality holds for sufficiently small $\delta$.

The argument for (P\ref{nonedge'}$'$), which is described below, is similar. Given a canonical copy $H_1$ of $H$ in ${\mathbb G}(H,n,p)$ and two non-adjacent vertices $x, y \in V(H_1)$, the expected number of accepted copies of $H$ whose intersection with $H_1$ is precisely $\{x,y\}$ is at most
$$
\rho n^{v(H) - 2} p^{e(H)} = \frac{f_H(n,p)}{{\mathbb E}(Y_H)} \cdot \frac{{\mathbb E}(Y_H)}{n^2} = \frac{f_H(n,p)}{n^2} \leq \delta \,. 
$$
Summing over all choices of $H_1$ and $x,y \in V(H_1)$, we conclude that the expected number of accepted copies $H_1$ of $H$ for which there is an accepted copy $H_2$ of $H$ which intersects $H_1$ on two non-adjacent vertices is at most
$$
{\mathbb E}(Z_H) \binom{v(H)}{2} \delta \leq f_H(n,p) v(H)^2 \delta \leq f_H(n,p)/12 \,,  
$$
where the last inequality holds for sufficiently small $\delta$. 

Finally, we consider (P\ref{few_copies_e'}$'$). Let $H_1$ be a canonical copy of $H$ in ${\mathbb G}(H,n,p)$ and let $e\in E(H_1)$. Given a proper subgraph $F$ of $H_1$ such that $e\in E(F)$, the expected number of accepted copies of $H$ whose intersection with $H_1$ is precisely $F$ is at most
$$ 
\rho n^{v(H) - v(F)} p^{e(H) - e(F)} = \frac{\rho {\mathbb E}(Y_H)}{{\mathbb E}(Y_F)} \le \frac{f_H(n,p)}{\hat f_H(n,p)}\le \frac{f_H(n,p)}{\varepsilon' n^2 p} \,,
$$
where the last inequality holds since the assumption $p \geq \varepsilon n^{-1/m_2(H)}$ implies that $\hat f_H(n,p) \ge \varepsilon' n^2 p$ for some constant $\varepsilon' > 0$ (which depends on $\varepsilon$ and $H$). Therefore, by Markov's inequality, the probability that $e$ belongs to more than $C f_H(n,p)/(n^2 p)$ accepted copies of $H$ is at most 
$$
\frac{2^{v(H)}f_H(n,p)/(\varepsilon' n^2 p)}{C f_H(n,p)/(n^2 p)} \leq \frac{2^{v(H)}}{\varepsilon' C} \,.
$$  
Summing over all choices of $H_1$ and $e \in E(H_1)$, we conclude that the expected number of accepted copies of $H$ which contain an edge that belongs to at least $C f_H(n,p)/(n^2 p)$ other accepted copies of $H$ is at most  
$$
\mathbb{E}(Z_H) \cdot e(H) \cdot \frac{2^{v(H)}}{\varepsilon' C} \leq f_H(n,p)/12 \,,
$$
provided that $C$ is sufficiently large.
 
We conclude that the family of copies of $H$ which remain in ${\mathcal A}$ after our deletion process satisfies Properties (P\ref{few_copies_e})--(P\ref{few_copies_2e})  
and so 
\begin{equation}\label{war0}
{\mathbb E} \left(\tS(n,p)\right) \geq f_H(n,p) - 4 \cdot f_H(n,p)/12 = 2 f_H(n,p)/3 \,.
\end{equation}

In order to prove that ${\mathbb E}(S_C(n,p)) \geq f_H(n,p)/2$ (thus completing the proof of the lemma), it remains to address Property (P\ref{no_edges}). This turns out to be a straightforward consequence of the fact that the number of edges of ${\mathbb G}(H, n, p)$ has the binomial distribution $Bin(e(H) n^2, p)$ and that, by Property (P\ref{few_copies_e'}), each edge of ${\mathbb G}(H, n, p)$ is contained in a limited number of copies from our family. Indeed, note that 
\begin{eqnarray}\label{war1}
\begin{aligned}
{\mathbb E} \left(\tS(n,p)\right) &= {\mathbb E} \left(\tS(n,p) \mid e \left(\mathbb G(H,n,p) \right) \leq C n^2p \right) \cdot \Pr \left(e \left(\mathbb G(H,n,p) \right) \leq C n^2 p \right) \\
&+ \sum_{M > Cn^2p} {\mathbb E} \left(\tS(n,p) \mid e \left(\mathbb G(H,n,p)\right) = M \right) \cdot \Pr \left(e \left(\mathbb G(H,n,p) \right) = M \right).
\end{aligned}
\end{eqnarray}

For the first part of the sum in~\eqref{war1} we have
\begin{equation}\label{war2}
\begin{aligned}
{\mathbb E} \big(\tS(n,p) &\mid e \left(\mathbb G(H,n,p)\right) \leq C n^2 p \big) \cdot \Pr \left(e \left(\mathbb G(H,n,p)\right) \leq Cn^2p \right) \\
& = {\mathbb E} \left(S_C(n,p) \mid e \left(\mathbb G(H,n,p)\right) \leq C n^2p\right) \cdot \Pr \left(e \left(\mathbb G(H,n,p)\right) \leq Cn^2p \right)\\
&\leq {\mathbb E} \left(S_C(n,p)\right).
\end{aligned}
\end{equation}

For the second part of the sum in~\eqref{war1}, it follows from (P\ref{few_copies_e'}) and from the fact that $e({\mathbb G}(H, n, p)) \sim Bin(e(H) n^2, p)$ that, for sufficiently large $C$, we have
\begin{eqnarray}\label{war3}
&&\sum_{M> Cn^2p}{\mathbb E}\left(\tS(n,p)\mid e\left(\mathbb G(H,n,p)\right) = M\right) \cdot \Pr \left(e \left(\mathbb G(H,n,p)\right) = M\right)\nonumber\\
&\le& \sum_{M> Cn^2p}{\mathbb E} \left(\frac{Cf_H(n,p)}{n^2p} \cdot e \left(\mathbb G(H,n,p) \right) \mid e \left(\mathbb G(H,n,p) \right) = M \right) \cdot \Pr \left(Bin \left(e(H) n^2, p \right) =  M \right) \nonumber\\
& = & \frac{C f_H(n,p)}{n^2p} \sum_{M > Cn^2p} M \cdot \Pr \left(Bin \left(e(H) n^2, p \right) = M \right) \leq \frac{f_H(n,p)}{6} \,.
\end{eqnarray}
Combining~\eqref{war0}, \eqref{war1}, \eqref{war2} and \eqref{war3}, we conclude that $\E(S_C(n,p))\ge f_H(n,p)/2$.   
\end{proof}

\bigskip

\begin{proof}{\textbf{of Lemma~\ref{lem::largeSparseFamilyBinomial}}}
For a constant $C$ let ${\mathcal F}_C$ be a largest family of canonical copies of $H$ in ${\mathbb G}(H,n,p)$ which satisfies Properties (P\ref{no_edges})--(P\ref{few_copies_2e}). Then $|{\mathcal F}_C| = S_C(n,p)$. In the remainder of this proof we will abbreviate  $S_C(n,p)$ under $S_C$.  

Let $e_1, \ldots, e_N$ denote the edges of the $n$th-blow-up of $H$. For every $1 \leq i \leq N$, let $Z_i$ be the indicator random variable for the event $e_i \in E({\mathbb G}(H,n,p))$ and let $c_i = C f_H(n,p)/(n^2 p)$. Clearly $S_C$ is a function of $Z_1, \ldots, Z_N$ and, by Property (P\ref{few_copies_e}), Part (a) of Theorem~\ref{th::Talagrand} is satisfied. Furthermore, by Property (P\ref{no_edges}) we can `certify' the existence of a family ${\mathcal F}_C$ as above by revealing at most $C n^2 p$ edges. Combined with Property (P\ref{few_copies_e}) and the choice of the $c_i$'s, setting $\psi \equiv C n^2 p (C f_H(n,p)/(n^2 p))^2$ we deduce that Part (b) of Theorem~\ref{th::Talagrand} is satisfied as well. 

Let $m$ be a median of $S_C$. It is known (cf. the comments following the statement of Theorem 2.29 in~\cite{JLR}) that if a random variable satisfies the assumptions of Theorem~\ref{th::Talagrand}, then the difference between its median and expectation is not greater than $\left(2\ln 2\sum_{k=1}^N c_k^2 \right)^{1/2}$. Therefore, for $S_C$ we have
\begin{eqnarray}\label{eq::med1}
\left|{\mathbb E}(S_C) - m \right| &\leq& \left(2\ln 2\sum_{k=1}^{e(H)n^2} c_k^2 \right)^{1/2} =
\left(2\ln 2\cdot e(H)n^2 \cdot \frac{C^2 f^2_H(n,p)}{n^4 p^2} \right)^{1/2} \nonumber\\
&\leq& 2 e(H) C \delta \cdot f_H(n,p) \leq f_H(n,p)/20,
\end{eqnarray}
where the second inequality holds since $p \geq (\delta n)^{-1}$ by assumption and the last inequality holds by taking $\delta$ to be sufficiently small with respect to $C$. Observe that the assumptions of Lemmata~\ref{lem::largeSparseFamilyBinomial} and~\ref{lem::fnp} are the same. Hence, for sufficiently small $\delta$ and sufficiently large $C$, we have 
\begin{equation} \label{eq::largeExpectation}
{\mathbb E}(S_C) \geq f_H(n,p)/2.
\end{equation} 
Using~\eqref{eq::largeExpectation} and (\ref{eq::med1}), we infer that  
\begin{equation}\label{eq::med2}
m \geq 0.9 {\mathbb E}(S_C).
\end{equation} 

Applying Theorem~\ref{th::Talagrand} with $r = m$ and $t = {\mathbb E}(S_C)/2$, and using the definition of a median and inequalities~\eqref{eq::largeExpectation} and~\eqref{eq::med2}, we obtain 
 
\begin{eqnarray*}
\Pr(S_C \leq f_H(n,p)/5) &\leq& \Pr(S_C \leq 2\,{\mathbb E}(S_C)/5)\\
& \leq& \Pr(S_C \leq m-{\mathbb E}(S_C)/2)\\ 
&\leq& 2 \exp \left(-\frac{ ({\mathbb E}(S_C)/2)^2}{4 (C n^2 p)(C f_H(n,p)/(n^2 p))^2} \right) \\
&\leq& 2 \exp \left(-\frac{ (f_H(n,p)/4)^2}{4 (C n^2 p)(C f_H(n,p)/(n^2 p))^2} \right) \\
&<& \exp \left(- \frac{n^2 p}{65 C^3} \right) \,. 
\end{eqnarray*}

Given a family ${\mathcal F}_C$ as above, one can construct a subfamily ${\mathcal F}'$ such that for every two distinct $H_1, H_2 \in {\mathcal F}'$ either $H_1 \cap H_2$ is the empty set or $H_1 \cap H_2\in\inter_H$. For this purpose, take an arbitrary $H_1 \in {\mathcal F}_C$, put it in ${\mathcal F}'$, delete from ${\mathcal F}_C$ all graphs $G$ whose intersection with $H_1$ is non-empty and not in $\inter_H$ and repeat this process until ${\mathcal F}_C = \emptyset$. It follows by Properties (P\ref{nonedge})--(P\ref{few_copies_2e}) that $|{\mathcal F}'| \geq \frac{1}{2^{v(H)}C + 1} |{\mathcal F}_C|$. Since, in particular, ${\mathcal F}'$ is a sparse $H$-family, the assertion of the lemma holds for $\alpha = e^{- 1/(65 C^3)}$ and $\beta = \frac{1}{5(2^{v(H)}C + 1)}$.  
\end{proof}

In the following two lemmata we describe two intervals such that, if $p$ is in one of these intervals, then, with very high probability, ${\mathbb G}(H,n,p)$ contains a sparse $H$-family with $\Theta \left(\E(Y_H) \right)$ copies of $H$.

\begin{lemma} \label{lem::manyHgnp}
For every $\varepsilon > 0$ and every graph $H$ such that $e(H) \geq 3$ and $K_{1,2} \subseteq H$, there exist positive constants $\alpha < 1$, $\beta$ and $\delta < 1$ such that the following holds. For every $n$ and $p$ such that
\begin{equation}\label{pcond}
\eps n^{-g_2(H)}\le p \le \delta n^{-g_1(H)},
\end{equation}
the probability that ${\mathbb G}(H,n,p)$ contains a sparse $H$-family with at least $\beta n^{v(H)}p^{e(H)}$ copies of $H$ is greater than $1 - \alpha^{n^2 p}$.
\end{lemma}

\begin{proof} 
If $g_2(H)\le g_1(H)$, then the assertion of the lemma trivially holds by taking $\delta < \eps$. Hence, for the remainder of the proof, we assume  that $g_1(H)< g_2(H)$. Since, moreover, $K_{1,2} \in \inter^c_H$, it follows by Lemma~\ref{lem::m2balanced}(\ref{g1less}) that $g_2(H) < 1$. Fix some $\eps>0$ and let $\alpha, \delta < 1$ be the constants whose existence is ensured by Lemma~\ref{lem::largeSparseFamilyBinomial}. Since $g_2(H) < 1$, it follows by the lower bound in~\eqref{pcond} that  
\begin{equation}\label{pbig}
p \geq \eps n^{-g_2(H)} > (\delta n)^{-1}
\end{equation}
for sufficiently large $n$.

Since $\delta < 1$, it follows by the upper bound in~\eqref{pcond} that
$$
p \leq \delta n^{-g_1(H)} \leq n^{-g_1(H)}.
$$

It then follows from Lemma~\ref{lem::propf}(\ref{fg1}) that $f_H(n,p)\ge n^{v(H)}p^{e(H)}$. Hence, by the definition of $f_H(n,p)$ we have 
\begin{equation}\label{HH0}
f_H(n,p)=n^{v(H)}p^{e(H)}.
\end{equation}

Furthermore, by the definition of $g_1(H)$, we have
$$
p \leq \delta n^{-g_1(H)} \leq \delta n^{-{(v(H)-2)/e(H)}},
$$
which by (\ref{HH0}) implies that
\begin{equation}\label{fup}
f_H(n,p) = n^{v(H)}p^{e(H)} \leq \delta n^2.
\end{equation}

Finally, observe that
\begin{equation}\label{pm2}
p \geq \eps n^{-g_2(H)} = \eps n^{-1/m_2(H)},
\end{equation}
where the last equality holds by Lemma~\ref{lem::m2balanced}(\ref{g2m2}).

Combining (\ref{pbig}), (\ref{fup}) and (\ref{pm2}) shows that all the conditions of Lemma~\ref{lem::largeSparseFamilyBinomial} are satisfied. Since $f_H(n,p) = n^{v(H)}p^{e(H)}$, this concludes the proof of the lemma.  
\end{proof}

\begin{lemma} \label{lem::manyHgnp2}
Suppose that $H$ is an $m_2$-balanced graph which is not a forest and let $\varepsilon > 0$ and $C \geq 1$ be constants. Then there exist positive constants $\alpha < 1$ and $\beta$ such that the following holds. For every $n$ and $p$ such that
\begin{equation}\label{pcond2}
\eps n^{-1/m_2(H)}\le p \le C n^{-1/m_2(H)},
\end{equation}
the probability that ${\mathbb G}(H,n,p)$ contains a sparse $H$-family with at least $\beta n^{v(H)}p^{e(H)}$ copies of $H$ is greater than $1 - \alpha^{n^2 p}$.
\end{lemma}

\begin{proof}
Fix some $\eps > 0$ and let $\alpha, \delta < 1$ be the constants whose existence is ensured by Lemma~\ref{lem::largeSparseFamilyBinomial} (for this $\varepsilon$). Note that $1/m_2(H) < 1$ since $H$ contains a cycle. Thus  
\begin{equation}\label{pbig2}
p \geq \eps n^{-1/m_2(H)} > (\delta n)^{-1}
\end{equation}
for sufficiently large $n$.

Since $H$ is $m_2$-balanced and $p \leq C n^{-1/m_2(H)}$, it follows by Lemma~\ref{lem::propf}(\ref{fm2}) that 
\begin{equation}\label{HH02}
f_H(n,p) \ge C^{- e(H)} n^{v(H)} p^{e(H)}.
\end{equation}

Furthermore
$$
p \leq C n^{-1/m_2(H)} = C n^{-(v(H)-2)/(e(H)-1)} < \delta n^{-(v(H)-2)/e(H)},
$$
where the last inequality holds for sufficiently large $n$. This implies that $n^{v(H)}p^{e(H)}\le \delta n^2$ and by the definition of $f_H(n,p)$ we conclude that
\begin{equation}\label{fup2}
f_H(n,p)\le n^{v(H)}p^{e(H)}\le \delta n^2.
\end{equation}

Combining (\ref{pcond2}), (\ref{pbig2}) and (\ref{fup2}) shows that all the conditions of Lemma~\ref{lem::largeSparseFamilyBinomial} are satisfied. Since $f_H(n,p) \geq C^{- e(H)} n^{v(H)} p^{e(H)}$, this concludes the proof of the lemma.
\end{proof}

Recall that we are actually interested in properties of ${\mathbb G}(H,n,M)$ rather than ${\mathbb G}(H,n,p)$. Hence, we end this section by stating the ${\mathbb G}(H,n,M)$ analogues of Lemmata~\ref{lem::manyHgnp} and \ref{lem::manyHgnp2}. 

\begin{lemma} \label{lem::manyHgnm}
For every $\varepsilon > 0$ and every graph $H$ such that $e(H) \geq 3$ and $K_{1,2} \subseteq H$, there exist positive constants $\alpha < 1$, $\beta$ and $\delta < 1$ such that the following holds. For every $n$ and $M$ such that
$$
\eps n^{2-g_2(H)}\le M \le \delta n^{2-g_1(H)},
$$ 
the probability that ${\mathbb G}(H,n,M)$ contains a sparse $H$-family with at least $\beta \E(X_H)$ copies of $H$ is greater than $1 - \alpha^{M}$.
\end{lemma}

\begin{lemma} \label{lem::manyHgnm2}
Suppose that $H$ is an $m_2$-balanced graph which is not a forest and let $\varepsilon > 0$ and $C \geq 1$ be constants. Then there exist positive constants $\alpha < 1$ and $\beta$ such that the following holds. For every $n$ and $M$ such that
$$
\eps n^{2 - 1/m_2(H)} \leq M \leq C n^{2 - 1/m_2(H)},
$$ 
the probability that ${\mathbb G}(H,n,M)$ contains a sparse $H$-family with at least $\beta \E(X_H)$ copies of $H$ is greater than $1 - \alpha^{M}$.
\end{lemma}

\section{Winning criteria for the $H$-game} \label{sec::eveyH}

In this section we state and prove three useful corollaries of the results proven in previous sections. Each of these corollaries provides a different sufficient condition for Waiter to force Client to build many copies of $H$ in a $(b:1)$ Waiter-Client game on the edge set of ${\mathbb B}_H(V_1, \ldots, V_{v(H)};s)$. In particular, Theorem~\ref{th::largeb} will readily follow from Theorem~\ref{th::BESH}(i) and the first result of this section. 

\begin{corollary} \label{cor::denseH}
Let $H$ be a graph with at least two edges. Then there exist positive constants $c$ and $\beta$ such that for every sufficiently large integer $s$ the following holds. If  
$$
2 \leq b+1 \leq c s^{1/m_2(H)} \,,
$$
then, playing a $(b:1)$ Waiter-Client game on the edge set of ${\mathbb B}_H(V_1,\ldots,V_{v(H)};s)$, Waiter has a strategy to force Client to build at least $\beta s^{v(H)}(b+1)^{-e(H)}$ copies of $H$.
\end{corollary}

\begin{proof}
Without loss of generality we can clearly assume that $H$ has no isolated vertices. If $H$ is a matching of size $k \geq 2$, then $b = O(s^2)$ by assumption. Suppose that $E(H) = \{x_i y_i : 1 \leq i \leq k\}$ and let $X_1, Y_1, \ldots, X_k, Y_k$ be the corresponding vertex sets in ${\mathbb B}_H(V_1,\ldots,V_{v(H)};s)$. For every $1 \leq i \leq k$, by offering only the edges of $E(X_i, Y_i)$, Waiter forces Client to claim $\Theta(s^2/b)$ of these edges. This yields $\Theta(s^{2k}/b^k)$ copies of $H$ in Client's graph. In the remainder of this proof we therefore assume that $H$ contains two adjacent edges.

For positive constants $\alpha < 1$ and $c < (1 - \alpha)/2$, let $2 \leq b+1 \leq c s^{1/m_2(H)}$, $N = e(H)s^2$, $M = M(s) = \lfloor (1-\alpha)N/(b+1) \rfloor$ and $p = p(s) = M/N$. Since $H$ is not a matching, $m_2(H) \geq 1$ and thus $b+1 \leq cs$. In particular 
\begin{equation} \label{eq::forLemmaDenseM}
M=\omega(1)\quad\text{and}\quad M \leq e(H) s^2/2 \,.
\end{equation}

A simple calculation shows that 
$$
\hat{f}_H(s,M) \geq \gamma \hat{f}_H(s,p)
$$
holds for some constant $\gamma > 0$ (which depends on $H$).

Moreover
$$
p = \frac{M}{N} \geq \frac{1-\alpha}{2(b+1)}\ge \frac{1-\alpha}{2c s^{1/m_2(H)}}\ge \frac{1}{s^{1/m_2(H)}}\,,
$$
where the last inequality holds since $c < (1-\alpha)/2$.

Hence, we can apply Lemma~\ref{lem::propf}(\ref{hatdense}) and obtain
\begin{equation*} 
\hat{f}_H(s,p) = s^2 p = \frac{M}{e(H)} \,.
\end{equation*}

Observe that the expected number of canonical copies of $H$ in ${\mathbb G}(H,s,M)$ is at least $2\delta s^{v(H)}p^{e(H)}$ for some positive constant $\delta$ (which depends on $H$). Let $\mathcal{G}$ denote the family of all subgraphs of ${\mathbb B}_H(V_1, \ldots, V_{v(H)};s)$ with precisely $M$ edges and let $\mathcal{F}$ denote the family of all graphs in $\mathcal{G}$ each containing at least $\delta s^{v(H)} p^{e(H)}$ copies of $H$. Since $\delta s^{v(H)} p^{e(H)} \leq {\mathbb E}(X_H)/2$ and~\eqref{eq::forLemmaDenseM} hold, the conditions of Lemma~\ref{lem::denseM} are satisfied. Therefore, there exists a positive constants $c'$ such that 
\begin{eqnarray}
|\mathcal{F}| &\geq& 
\left(1 - \exp(-c' \hat{f}_H(s,M))\right) |\mathcal{G}| \nonumber \\
&\geq& \left(1 - \exp(-c'\gamma \hat{f}_H(s,p))\right) |\mathcal{G}| \nonumber \\
&=& \left(1 - \exp(-c'\gamma M/e(H))\right) |\mathcal{G}| \nonumber \\
&>& \left(1 - \alpha^M \right) |\mathcal{G}|,
\end{eqnarray}
provided that $\alpha>\exp(-c'\gamma/e(H))$.

Let $X = E({\mathbb B}_H(V_1, \ldots, V_{v(H)};s))$. Then $N = |X|$, $\mathcal{F}$ is $M$-uniform, $|\mathcal{F}| > (1 - \alpha^M)|\mathcal{G}|$ and $b+1 \leq (1 - \alpha) N/M$. It thus follows by Theorem~\ref{th::smallFamily} that, playing the $(b:1)$ Waiter-Client game $(X, {\mathcal F})$, Waiter has a strategy to force Client to fully claim some $A \in {\mathcal F}$ during the first $M$ rounds of the game. Hence, after $M$ rounds, the subgraph of ${\mathbb B}_H(V_1, \ldots, V_{v(H)};s)$ built by Client contains at least $\delta s^{v(H)}p^{e(H)} \geq \beta s^{v(H)}(b+1)^{-e(H)}$ copies of $H$, for some constant $\beta > 0$.   
\end{proof}

\begin{corollary} \label{cor::manyHnew}
Let $H$ be a graph such that $e(H) \geq 3$ and $K_{1,2} \subseteq H$. Then there exist positive constants $c_1$, $c_2$ and $\beta'$ such that for every sufficiently large integer $s$ the following holds. Suppose that $b = b(s)$  is a positive integer satisfying  
$$
c_1 s^{g_1(H)} \leq b+1 \leq c_2 s^{g_2(H)}.
$$
Then, playing a $(b:1)$ Waiter-Client game on the edge set of ${\mathbb B}_H(V_1,\ldots,V_{v(H)};s)$, Waiter has a strategy to force Client to build a sparse $H$-family consisting of at least $\beta' s^{v(H)}(b+1)^{-e(H)}$ copies of $H$.
\end{corollary}

\begin{proof}
If $g_1(H) \geq g_2(H)$, then our claim trivially follows by taking $c_2 < c_1$. We can thus assume that $g_1(H) < g_2(H)$ which, by Lemma~\ref{lem::m2balanced}(\ref{g1less}), implies that $g_2(H) < 1$. Let $\alpha < 1$, $\beta$ and $\delta$ be the constants whose existence follows from Lemma~\ref{lem::manyHgnm}, when applied with $\eps=1$. Fix
$$
c_1=(1-\alpha)e(H)/\delta\quad\text{and}\quad c_2=(1-\alpha)e(H)/2
$$
Let $M = \lfloor (1-\alpha)e(H)s^2/(b+1) \rfloor$. Since $g_2(H) < 1$, it follows that $b = o(s)$ and $M = \omega(1)$. Furthermore 
$$
M>\frac{1-\alpha}{2(b+1)}\,e(H)\,s^2\ge \frac{1-\alpha}{2c_2}\,e(H)\,s^{2-g_2(H)}= s^{2-g_2(H)}
$$
and
$$
M\le\frac{1-\alpha}{b+1}\,e(H)\,s^2\le \frac{1-\alpha}{c_1}\,e(H)\,s^{2-g_1(H)}=\delta s^{2-g_1(H)}.
$$ 
Thus, by Lemma~\ref{lem::manyHgnm}, the probability that ${\mathbb G}(H,n,M)$ contains a sparse $H$-family with at least $\beta \E(X_H)$ copies of $H$ is greater than $1 - \alpha^{M}$. 

Let $\mathcal{G}$ denote the family of all subgraphs of ${\mathbb B}_H(V_1, \ldots, V_{v(H)};s)$ with precisely $M$ edges and let $\mathcal{H}$ be the family of edge sets of all graphs $G \in \mathcal{G}$ which contain a sparse $H$-family consisting of at least $\beta \E(X_H)$ copies of $H$. By the argument above, we have $|\mathcal{H}|>(1 - \alpha^M) |\mathcal{G}|$. 

Let $X = E({\mathbb B}_H(V_1, \ldots, V_{v(H)};s))$ and $N = e(H) s^2$. It is easy to see that $b+1 \leq (1 - \alpha) N/M$. Hence, applying Theorem~\ref{th::smallFamily} to the family ${\mathcal H}$, shows that, playing a $(b:1)$ Waiter-Client game on the edge set of ${\mathbb B}_H(V_1, \ldots, V_{v(H)};s)$, Waiter can force Client to build a sparse $H$-family consisting of at least $\beta \E(X_H) \geq \beta' s^{v(H)} (b+1)^{-e(H)}$ copies of $H$. 
\end{proof}

\begin{corollary} \label{cor::manyHnew2}
Let $H$ be an $m_2$-balanced graph which is not a forest. Then there exist positive constants $c_1 < c_2$ and $\beta'$ such that for every sufficiently large integer $s$ the following holds. Suppose that $b = b(s)$ is a positive integer satisfying  
$$
c_1 s^{1/m_2(H)} \leq b+1 \leq c_2 s^{1/m_2(H)}.
$$ 
Then, playing a $(b:1)$ Waiter-Client game on the edge set of ${\mathbb B}_H(V_1,\ldots,V_{v(H)};s)$, Waiter has a strategy to force Client to build a sparse $H$-family consisting of at least $\beta' s^{v(H)}(b+1)^{-e(H)}$ copies of $H$.
\end{corollary}

\begin{proof}
Let $\alpha < 1$ and $\beta$ be the positive constants whose existence is ensured by Lemma~\ref{lem::manyHgnm2} with $C=1$ and $\eps=1/4$. Let
$$
c_1 = (1-\alpha) e(H) \quad \text{and} \quad c_2 = 2(1-\alpha)e(H)
$$
and suppose that $c_1 s^{1/m_2(H)} \leq b+1 \leq c_2 s^{1/m_2(H)}$. Let $M = \lfloor (1-\alpha)e(H)s^2/(b+1) \rfloor$. Since $H$ contains a cycle, it follows that $m_2(H) > 1$ and thus $b = o(s)$ and $M = \omega(1)$. Furthermore 
$$
M>\frac{1-\alpha}{2(b+1)}\,e(H)\,s^2\ge \frac{1-\alpha}{2c_2}\,e(H)\,s^{2-1/m_2(H)} = s^{2-1/m_2(H)}/4
$$
and
$$
M\le\frac{1-\alpha}{b+1}\,e(H)\,s^2\le \frac{1-\alpha}{c_1}\,e(H)\,s^{2-1/m_2(H)}=s^{2-1/m_2(H)}.
$$ 
Thus, by Lemma~\ref{lem::manyHgnm2}, the probability that ${\mathbb G}(H,n,M)$ contains a sparse $H$-family with at least $\beta \E(X_H)$ copies of $H$ is greater than $1 - \alpha^{M}$. 

Similarly to the proof of Corollary~\ref{cor::manyHnew}, we can now apply Theorem~\ref{th::smallFamily} with appropriate parameters to conclude that Waiter can force Client to build a sparse $H$-family consisting of at least $\beta' s^{v(H)}(b+1)^{-e(H)}$ copies of $H$.
\end{proof}

\section{Clique games} \label{sec::cliques}

In this section we will apply the general method we developed in the previous three sections to the special case of cliques. Before proving Theorem~\ref{th::manyK} we need one more lemma.

\begin{lemma} \label{lem::balan}
Let $k \geq 3$ be an integer and let $M_1 = \{e_1, \ldots e_{\lfloor k/2 \rfloor}\}$ and $M_2 = \{e_{\lfloor k/2 \rfloor + 1}, \ldots e_{k-2}\}$ be two edge disjoint matchings of $K_k$. For every $1 \leq i \leq k-2$, let $H_i = K_k \setminus \{e_1, \ldots, e_i\}$.
\begin{enumerate}[(i)]  
\item 
For every $1 \leq i \leq k-2$, if $k \neq 5$ or $i \neq 3$, then $g_1(H_i) = (v(H_i)-2)/e(H_i)$.
\item
$g_2(H_i) = (v(H_i)-2)/(e(H_i)-1)$ for every $1 \leq i \leq k-2$.
\end{enumerate}
\end{lemma}  

\begin{proof}
Starting with (i), let us fix an integer $k \geq 3$. It is an immediate consequence of the definition of $g_1(H)$ that $g_1(H_i) \geq (v(H_i)-2)/e(H_i)$ for every $1 \leq i \leq k-2$. Suppose that $g_1(H_i) > (v(H_i)-2)/e(H_i)$ holds for some $1 \leq i \leq k-2$. It thus follows by the definition of $g_1$ that there exists some non-complete graph $H' \subsetneq H_i$ with at least two edges such that 
\begin{equation} \label{eq::t1}
\frac{v(H_i) - 2}{e(H_i)} < \frac{v(H_i) - v(H')}{e(H_i) - e(H')} \,.
\end{equation}
A straightforward calculation  shows that
\begin{equation} \label{eq::t2}
\frac{e(H')}{v(H')-2} > \frac{e(H_i)}{v(H_i)-2} \,.
\end{equation}
Let $t = v(H')$. By the definitions of $H_i$ and $H'$ and by~\eqref{eq::t2} we have $e(H')\leq \binom t2-1$, $v(H_i) = k$, $e(H_i) = \binom{k}{2} - i$ and $k > t \geq 3$. Hence \eqref{eq::t2} implies that
$$ 
\frac{\binom{t}{2} - 1}{t-2} > \frac{\binom{k}{2} - i}{k-2} \,,
$$
and thus 
\begin{equation} \label{eq::tnew}
\frac{t+1}{2} > \frac{k+1}{2} - \frac{i-1}{k-2} \,.
\end{equation}
Since $i \leq k-2$ and $t < k$, it follows that
$$
k-1 \geq t > k - \frac{2(i-1)}{k-2} > k-2 \,,
$$
where the second inequality holds by~\eqref{eq::tnew}. These inequalities can only hold if
\begin{equation} \label{eq::t3}
t = k-1 \textrm{  and  } i > k/2 \,.
\end{equation}

Consequently, $v(H_i) - v(H') = 1$ and, since $\delta(H_i) \geq k-3$, we have $e(H_i) - e(H') \geq k-3$. Thus, by~\eqref{eq::t1} we have 
\begin{equation}\label{eq::next}
k-3 \leq e(H_i) - e(H') = \frac{e(H_i)-e(H')}{v(H_i)-v(H')} < \frac{e(H_i)}{v(H_i)-2} = \frac{k+1}{2} - \frac{i-1}{k-2} \,.
\end{equation}
It is easy to verify that the set of inequalities~\eqref{eq::t3}, \eqref{eq::next} and $i\leq k-2$
is satisfied if and only if $k=5$, $i=3$, and $H' = K_4 \setminus \{e\}$ for some edge $e$.  

Summarizing, we have shown that, if $k \neq 5$, then $g(H_i) = (v(H_i)-2)/e(H_i)$ holds for every $1 \leq i \leq k-2$, and if $k = 5$, then then $g(H_i) = (v(H_i)-2)/e(H_i)$ holds for every $i \in [k-2] \setminus \{3\}$. This proves (i).

Next, we prove (ii). It is easy to verify that our claim holds for $k \in \{3,4\}$; hence, from now on we assume that $k \geq 5$. For every $1 \leq i \leq k-2$, let $t_i^*$ denote the largest order of a clique in $H_i$. It is easy to see that $t_i^* = k - i$ for every $1 \leq i \leq k/2$ and that $t_i^* = k - \lfloor k/2 \rfloor$ for every $k/2 < i \leq k - 2$. In particular, since $k \geq 5$, we have $t_i^* \geq 3$ for every $1 \leq i \leq k-2$ and so
$$
\min \left\{\frac{v(K_t)-2}{e(K_t)-1} : K_t\subsetneq H_i, t \geq 3\right\} = \frac{2}{t_i^* + 1}.
$$
By the definition of $g_2$ we then have   
$$
g_2(H_i) = \min\left\{\frac{k-2}{\binom{k}{2} - i - 1}, \frac{2}{t_i^* + 1}\right\}.
$$

A straightforward calculation shows that, since $k \geq 5$, we have
$$
\frac{k-2}{\binom{k}{2} - i - 1} \leq \frac{2}{k-i+1}
$$ 
for every $1 \leq i \leq k/2$ and 
$$
\frac{k-2}{\binom{k}{2} - i - 1} \leq \frac{2}{k - \lfloor k/2 \rfloor + 1}
$$ 
for every $k/2 < i \leq k-2$.

We conclude that  
$$
g_2(H_i) = \frac{k-2}{\binom{k}{2} - i - 1} = \frac{v(H_i) - 2}{e(H_i) - 1} 
$$
holds for every $1 \leq i \leq k-2$ as claimed. 
\end{proof}

We can now prove the main result of this section.

\begin{proof}{\textbf{of Theorem~\ref{th::manyK}}} 
For every positive integer $b$ and every $k \geq 3$, it follows by Theorem~\ref{th::BESH}(i) that $S(K_k, n, b) = O \left(n^k \cdot (b+1)^{- \binom{k}{2}} \right)$. This proves the upper bound in~\ref{th::manyK}(i).  

The remainder of this proof is dedicated to the lower bounds. Starting with (i), assume that $k\neq 5$. Assume further that $k \geq 4$ (for technical reasons, we will handle the simple case $k=3$ separately). Let $s = \lfloor n/k \rfloor$ and let $G = (V,E)$ be a copy of ${\mathbb B}_{K_k}(V_1, \ldots, V_k;s)$ in $K_n$. It clearly suffices to prove that Waiter can force Client to build the required number of copies of $K_k$ when playing on $E$. 

Since $m_2(K_k) = \left(\binom{k}{2} - 1 \right)/(k-2)$, it follows by Corollary~\ref{cor::denseH} that there are positive constants $c_0$ and $\beta$ such that if 
\begin{equation} \label{eq::b0}
2 \leq b+1 \leq c_0 s^{(k-2)/\left(\binom{k}{2} - 1 \right)} \,,
\end{equation}
then $S(K_k, n, b) \geq \beta s^k (b+1)^{- \binom{k}{2}} = \Theta \left(n^k \cdot (b+1)^{- \binom{k}{2}} \right)$. 

Hence, from now on we assume that $b = \Omega \left(s^{(k-2)/\left(\binom{k}{2} - 1 \right)} \right)$. For every $1 \leq i \leq k-2$, let $e_1, \ldots, e_i$ and $H_i$ be as in Lemma~\ref{lem::balan}. 

It follows by Corollary~\ref{cor::manyHnew} that for every $1 \leq i \leq k-2$ there are positive constants $c_1^i$, $c_2^i$ and $\beta_i$ such that if 
\begin{equation} \label{eq::prevHi}
c_1^i s^{g_1(H_i)} \leq b+1 \leq c_2^i s^{g_2(H_i)} \,,
\end{equation}
then, playing on $E$, Waiter can force Client to build a sparse $H_i$-family $\mathcal{A}_i$ consisting of at least $\beta_i s^{v(H_i)} (b+1)^{-e(H_i)}$ copies of $H_i$.

For every $1 \leq i \leq k-2$ let $e_i = u_i u'_i$ and let $U_i$ and $U'_i$ denote the corresponding pairs of vertex sets in ${\mathbb B}_{K_k}(V_1, \ldots, V_k;s)$. Fix some $1 \leq i \leq k-2$ for which~\eqref{eq::prevHi} is satisfied (assuming such an $i$ exists). We are now ready to describe Waiter's strategy (for this $i$); it is divided into the following two stages. 

\bigskip

\noindent \textbf{Stage I:} Without offering any edges of $\bigcup_{j=1}^i E_G(U_j, U'_j)$, Waiter forces Client to build a sparse $H_i$-family $\mathcal{A}_i$ consisting of at least $\beta_i s^{v(H_i)} (b+1)^{-e(H_i)}$ copies of $H_i$.        

\bigskip

\noindent \textbf{Stage II:} This stage is further divided into $i$ phases. Let $\mathcal{A}_i^0 = \mathcal{A}_i$ and, for every $1 \leq j \leq i$, at the end of Phase $j$, let ${\mathcal A}_i^j$ denote the family of canonical copies of $H_i \cup \{e_1, \ldots, e_j\}$ in Client's graph. For every $1 \leq j \leq i$, the $j$th phase lasts exactly $\lfloor |\mathcal{A}_i^{j-1}|/(b+1) \rfloor$ rounds. For every $1 \leq j \leq i$ and every $\ell$ which satisfies 
$$
1 \leq \ell \leq \left\lfloor \frac{|\mathcal{A}_i^{j-1}|}{b+1}\right\rfloor,
$$ 
in the $\ell$th round of Phase $j$, Waiter offers Client $b+1$ free edges $x_1 x'_1, \ldots, x_{b+1} x'_{b+1} \in E_G(U_j, U'_j)$, such that for every $1 \leq t \leq b+1$, $x_t$ and $x'_t$ are vertices of a copy of $H_i$ in $\mathcal{A}_i^{j-1}$, corresponding to the vertices $u_i$ and $u'_i$. 

\bigskip

As noted above, Corollary~\ref{cor::manyHnew} implies that Waiter can play according to Stage I of the proposed strategy. Moreover, since no two graphs in $\mathcal{A}_i$ share a pair of non-adjacent vertices, Waiter can play according to Stage II of the proposed strategy as well. It remains to show that, by doing so, he forces Client to build $\Omega \left(n^k \cdot (b+1)^{-\binom{k}{2}} \right)$ copies of $K_k$. Since $H_i \cup \{e_1, \ldots, e_i\} \cong K_k$, it suffices to prove that $|\mathcal{A}_i^i| \geq \Omega \left(n^k \cdot (b+1)^{-\binom{k}{2}} \right)$. Since, for every $1 \leq j \leq i$, no two graphs in $\mathcal{A}_i^{j-1}$ share a pair of non-adjacent vertices, it follows by the description of Stage II of the proposed strategy that $|\mathcal{A}_i^j| \geq \lfloor |\mathcal{A}_i^{j-1}|/(b+1) \rfloor \geq |\mathcal{A}_i^{j-1}|/(b+1) - 1$. Therefore  
$$
|\mathcal{A}_i^i| \geq \frac{|\mathcal{A}_i|}{(b+1)^i} - i \geq \frac{\beta_i s^{v(H_i)}}{(b+1)^{e(H_i)+i}} - i = \frac{\beta_i s^{k}}{(b+1)^{\binom{k}{2}}} - i = \Theta \left(n^k(b+1)^{-\binom{k}{2}} \right) \,.
$$

\bigskip      

Note that by Lemma~\ref{lem::balan} the inequalities~\eqref{eq::prevHi} are equivalent to
\begin{equation} \label{eq::Hi}
c_1^i s^{(v(H_i)-2)/e(H_i)} \leq b+1 \leq c_2^i s^{(v(H_i)-2)/(e(H_i)-1)} \,.
\end{equation}
We have thus proved that Waiter can force Client to build the required number of copies of $K_k$ provided that $b$ satisfies~\eqref{eq::b0} or~\eqref{eq::Hi} for some $1 \leq i \leq k-2$, that is, provided that
\begin{eqnarray*} 
b+1 &\in& \left[2, c_0 s^{(k-2)/\left(\binom{k}{2} - 1 \right)} \right] 
\cup \bigcup_{i=1}^{k-2} \left[c_1^i s^{(v(H_i)-2)/e(H_i)}, c_2^i s^{(v(H_i)-2)/(e(H_i)-1)} \right] \\
&=& \left[2,c_0 s^{(k-2)/\left(\binom{k}{2} - 1 \right)} \right] \cup \bigcup_{i=1}^{k-2} \left[c_1^i s^{(k-2)/\left(\binom{k}{2} - i \right)}, c_2^i s^{(k-2)/\left(\binom{k}{2} - i - 1 \right)} \right]\,.
\end{eqnarray*}

Since for $i = k-2$ we have $s^{(k-2)/\left(\binom{k}{2} - i - 1 \right)} = s^{2/(k-1)}$, we conclude that there exists a positive constant $c$ for which the above union of intervals of values of $b+1$ covers the interval $\left[2, c_2^{k-2} s^{2/(k-1)} \right] = \left[2, c n^{2/(k-1)} \right]$ with the possible exception of the intervals 
$$
\left(c_i n^{(k-2)/\left(\binom{k}{2} - i \right)}, C_i n^{(k-2)/\left(\binom{k}{2} - i \right)} \right)
$$ 
for some values of $1 \leq i \leq k-2$ and some positive constants $c_i < C_i$. However, by the bias-monotonicity of Waiter-Client games, we have $S(K_k, n, b) \geq S(K_k, n, b')$ if $b \leq b'$. Therefore, if $b$ is of order $n^{(k-2)/\left(\binom{k}{2} - i \right)}$ for some $1 \leq i \leq k-2$ and $b' > b$ is such that $b' + 1 = \left \lceil C_i n^{(k-2)/\left(\binom{k}{2} - i\right)} \right\rceil$, then 
$$
S(K_k, n, b) \geq S(K_k, n, b') = \Omega \left(n^k \cdot (b'+1)^{- \binom{k}{2}} \right) = \Omega \left(n^k \cdot (b+1)^{- \binom{k}{2}} \right) \,.
$$

Now, assume that $k = 3$ and let $b \leq c n$, where $c > 0$ is a sufficiently small constant. We will describe a strategy for Waiter to force Client to build $\Theta(n^3/b^3)$ copies of $K_3$ in the $(b : 1)$ game on ${\mathbb B}_{K_3}(V_1, V_2, V_3; s)$, where $s = \lfloor n/3 \rfloor$. This strategy consists of the following three stages.  

\bigskip

\noindent \textbf{Stage I:} Let $V_1 = \{x_1, \ldots, x_s\}$. For every $1 \leq i \leq s$ and $1 \leq j \leq \lfloor s/(b+1) \rfloor$, in the $((i-1) \lfloor s/(b+1) \rfloor + j)$th round, Waiter offers Client $b+1$ arbitrary free edges of $E(x_i, V_2)$. Waiter then proceeds to Stage II.        

\bigskip

\noindent \textbf{Stage II:} Let $V_2 = \{y_1, \ldots, y_s\}$. For every $1 \leq i \leq s$ and $1 \leq j \leq \lfloor s/(b+1) \rfloor$, in the $((i-1) \lfloor s/(b+1) \rfloor + j)$th round of this stage, Waiter offers Client $b+1$ arbitrary free edges of $E(y_i, V_3)$. Waiter then proceeds to Stage III.       

\bigskip

\noindent \textbf{Stage III:} For every $x \in V_1$ and $z \in V_3$ let $t(xz)$ denote the number of vertices $y \in V_2$ such that both $xy$ and $yz$ were claimed by Client. Let $e_1, \ldots, e_{s^2}$ be a linear ordering of the edges of $E(V_1, V_3)$ such that, for every $1 \leq i \neq j \leq s^2$, $e_i < e_j$ if and only if $t(e_i) \geq t(e_j)$. For every $1 \leq \ell \leq \lfloor s^2/(b+1) \rfloor$, in the $\ell$th round of this stage, Waiter offers Client the edges $e_{(\ell-1) (b+1) + 1}, \ldots, e_{\ell (b+1)}$.          

\bigskip

At the end of Stage II, every vertex of $V_1$ has $\lfloor s/(b+1) \rfloor$ neighbours in $V_2$ and every vertex of $V_2$ has $\lfloor s/(b+1) \rfloor$ neighbours in $V_3$. Hence, Client's graph contains $s \cdot \left(\lfloor s/(b+1) \rfloor \right)^2 \geq s^3/(5 b^2)$ canonical copies of $K_{1,2}$. Note that $t(e) \leq \lfloor s/(b+1) \rfloor \leq s/b$ for every $e \in E(V_1, V_3)$. Hence, by rejecting $b$ of the edges $e_1, \ldots, e_{b+1}$ in the first round of Stage III, Client avoids closing at most $b \cdot s/b = s$ copies of $K_{1,2}$ into triangles. Moreover, for every $1 \leq i < j \leq \lfloor s^2/(b+1) \rfloor$, if $e_i$ is the edge Client claims in the $i$th round of Stage III and $e_j$ is some edge Waiter offers Client in the $j$th round of Stage III, then $t(e_i) \geq t(e_j)$. We conclude that the number of (canonical) triangles Client is forced to build is at least 
$$
\frac{1}{b+1} \cdot \left(s^3/(5 b^2) - s \right) = \Theta(n^3/b^3) \,,
$$    
where the last equality holds for sufficiently small $c$.

This concludes the proof of Part (i).

The proof of the lower bound in Part (ii) is based on a similar idea, but is actually much simpler. Since we just want to prove that $S(K_5, n, b) > 0$, we can use the bias monotonicity of Waiter-Client games to assume that $b = c n^{2/(5-1)} = c n^{1/2}$. Waiter follows the same strategy as in Part (i) (for $k \geq 4$) with $i = k-2 = 3$. It is easy to verify that $H_3$ is $m_2$-balanced and that $n^{1/m_2(H_3)} = n^{1/2}$. It thus follows by Corollary~\ref{cor::manyHnew2} that Waiter can follow Stage I of the proposed strategy. Using the same arguments as in the proof of Part (i), one can show that Waiter can follow Stage II of the proposed strategy as well and that, by doing so, he forces at least one copy of $K_5$ in Client's graph. 
\end{proof}

Finally, let us comment on the technical problems arising when one tries to prove the lower bound in (i) for $k=5$. A careful analysis of the proofs of this section shows that, for $K_5$, Conjecture~\ref{conj} is true for almost every value of $b$. Unfortunately, our method fails if 
$$
\Omega(n^{(v(H_3)-2)/e(H_3)}) = \Omega(n^{3/7}) = b = o(n^{1/m(K_5)}) = o(n^{1/2}).
$$
The problem is that, for these values of $b$, we cannot find a ``good'' spanning subgraph of $K_5$. Although we can, in fact, fill in the gap using rather complicated ad hoc methods, we could not find any elegant and compact way to adjust our general method to handle this case as well.

\section*{Acknowledgment}

We would like to thank the anonymous referees for helpful comments.

\end{document}